\documentclass{amsart}
\usepackage{geometry}
\usepackage{hyperref}
\usepackage{arydshln}
\usepackage{cases}
\usepackage{amsmath}
\usepackage{amsfonts}
\usepackage{bm}
\usepackage{arydshln}
\usepackage{mathrsfs}
\usepackage{pb-diagram}
\usepackage{amssymb}
\usepackage{xypic}
\usepackage[dvips]{graphicx}
\usepackage[all]{xy}
\usepackage{CJK}
\usepackage[dvipsnames]{xcolor}
\usepackage{mathtools}
\usepackage{setspace}
\usepackage{enumerate}
\newtheorem{theorem}{Theorem}[section]
\newtheorem{proposition}[theorem]{Proposition}
\newtheorem{lemma}[theorem]{Lemma}
\newtheorem{condition}[theorem]{Condition}
\theoremstyle{definition}
\newtheorem{definition}[theorem]{Definition}
\newtheorem{example}[theorem]{Example}
\newtheorem{corollary}[theorem]{Corollary}

\theoremstyle{remark}
\newtheorem{remark}[theorem]{Remark}
\newtheorem{Notation}[theorem]{Notation}

\numberwithin{equation}{section}
\begin{document}

\title{Some elementary properties of Laurent phenomenon algebras}

\author{Qiuning Du and Fang Li}
\address{Department
of Mathematics, Zhejiang University (Yuquan Campus), Hangzhou, Zhejiang
310027,  PR China
}
\email{11735007@zju.edu.cn, \;fangli@zju.edu.cn}
\date{version of \today}

\dedicatory{}

\thanks{\textit{Mathematics Subject Classification(2020): 13F60, 13F65}}
\keywords{Laurent phenomenon algebra, seed, Laurent polynomial}

\begin{abstract}
Let $\Sigma$  be Laurent phenomenon (LP) seed of rank $n$, $\mathcal{A}(\Sigma)$, $\mathcal{U}(\Sigma)$ and $\mathcal{L}(\Sigma)$ be its corresponding Laurent phenomenon algebra, upper bound and lower bound respectively.
We prove that each seed of $\mathcal{A}(\Sigma)$ is uniquely defined by its cluster, and any two seeds of  $\mathcal{A}(\Sigma)$ with $n-1$ common cluster variables are connected with each other by one step of mutation. The method in this paper also works for (totally sign-skew-symmetric) cluster algebras.
Moreover, we show that $\mathcal{U}(\Sigma)$ is invariant under seed mutations when each exchange polynomials coincides with its exchange Laurent polynomials of $\Sigma$.
Besides, we obtain the standard monomial bases of $\mathcal{L}(\Sigma)$.
We also prove that $\mathcal{U}(\Sigma)$ coincides with $\mathcal{L}(\Sigma)$ under certain conditions.
\end{abstract}
\maketitle
\tableofcontents
\section{Introduction}
Cluster algebras were introduced by Fomin and Zelevinsky in \cite{FZ}.
The core idea to define cluster algebra of rank $n$ is that one should have a {\em cluster seed} and an operator on cluster seeds, called {\em mutation}.
Roughly, a cluster seed $\Sigma_{t_0}$ is a collection of variables $x_{1;t_0},\cdots,x_{n;t_0}$ (cluster variables ) and binomials $F_{1;t_0},\cdots,F_{n;t_0}$ (exchange polynomials). One can apply mutation to a cluster seed to produce a new seed, i.e., new  variables and new binomials.
Note that the exchange polynomial in cluster algebra is always a binomial. One of the main results in cluster algebras is that they have Laurent phenomenon \cite{FZ}.

In the theory of cluster algebras, the following are interesting conjectures on seeds of cluster algebras: in a cluster algebra of rank $n$, (1) each seed is uniquely defined by its cluster; (2) any two seeds with $n-1$ common cluster variables are connected with each other by one step of mutation. One can refer \cite{CL,GSV} for detailed proof.

Significant notations of the upper cluster algebra, upper bound and lower bound associated with the cluster seed was introduced by Berenstein, Fomin and Zelevinsky to study the structure of cluster algebras in \cite{CA3}.
There are some theorems of upper bounds and lower bounds: (a) under the coprime condition, the upper bound is invariant under seed mutations; (b) the standard monomials in $x_1,x_1',\dots,x_n,x_n'$ are linearly independent over $\mathbb{ZP}$ if and only if the cluster seed is acyclic; (c) under the coprime and acyclic, then the upper bound coincides with the lower bound.

Muller showed that locally acyclic cluster algebras coincide with their upper cluster algebras in \cite{M}.
Gekhtman, Shapiro and Vainshte in \cite{GSV} proved (a) for generalized cluster algebras, then Bai, Chen, Ding and Xu demonstrated (c) and the sufficiency of (b) in \cite{BAI}. Besides, Bai discovered that acyclic generalized cluster algebras coincide with their generalized upper cluster algebras.

Laurent phenomenon (LP) algebras were introduced by Lam and Pylyavskyy in \cite{LP}, which generalize cluster algebras from the perspective of exchange relations. The exchange polynomials in LP algebras were allowed to have arbitrarily many monomials, rather than being just binomials. It turns out that the Laurent phenomenon also appears in LP algebras \cite{LP}.

One should note that our method also works for cluster algebras and generalized cluster algebras. We do not talk much about generalized cluster algebras in this paper, and one can refer \cite{BAI,CL,CS,GSV,NT} for details.

In this paper, we first affirm the conjectures on seeds of cluster algebras with respect to LP algebras.
\begin{theorem}
\label{conj} In a LP algebra of rank $n$,
\begin{enumerate}
  \item (Theorem \ref{mainthm1}) each LP seed is uniquely defined by its cluster.
  \item (Theorem \ref{mainthm2}) any two LP seeds with $n-1$ common cluster variables are connected with each other by one step of mutation.
\end{enumerate}
\end{theorem}

Second, we affirm theorems of upper bounds and lower bounds with respect to LP algebras under some conditions, by using the similar methods developed in \cite{CA3}.
\begin{condition}
Let $M_k$ be the lexicographically first monomial in the irreducible polynomial $F_k$ and $f_k(x_i)$ be the polynomial on $x_i$ in $R[x_2,\dots,\hat{x}_i,\dots,\hat{x}_k,\dots,x_n](x_i)$ without constant terms in $F_k$ for any $i\neq k$. Assume that for a LP seed $(\mathbf{x,F})$ of rank $n$, $\forall k\in [1,n]$, $F_k$ satisfies the following conditions:
\begin{enumerate}[(i)]
  \item $\hat{F}_k=F_k$.
  \item $M_k$ is of the form $\mathbf{x^{v_k}}=
  \begin{cases}
  x_{k+1}^{v_{k+1,k}}\cdots x_{n}^{v_{n,k}} & k\in[1,n-1] \\
  1 & k=n \end{cases}$, where $\mathbf{v_k}\in \mathbb{Z}_{\geq0}^{n-k}$ for $k\in[1,n-1]$.
  \item when $x_1\in F_k$ for $k\neq 1$, $F_k=M_k+f_k(x_1)$.
  \item when $x_1\notin F_k$ for $k\neq 1\ or\ 2$, if there exist an index $i$ in $[2,k-1]$ such that $x_k\in M_i$, then $F_k=M_k+f_k(x_i)$.
\end{enumerate}
\label{condition}
\end{condition}

\begin{theorem}
\begin{enumerate}[(a)]
  \item (Theorem \ref{thma}) Under (i) of Condition \ref{condition}, the upper bound is invariant under LP mutations.
  \item (Theorem \ref{wuguan}) Under (i) and (ii) of Condition \ref{condition},the standard monomials in $x_1,x_1',\dots,x_n,x_n'$ form an $R$-basis for $\mathcal{L}(\Sigma)$.
  \item (Theorem \ref{Thmc}) Under Condition \ref{condition}, the upper bound coincides with the lower bound.
\end{enumerate}
\label{abc}
\end{theorem}

This paper is organized as follows: in Section 2, some basic definitions are given. In Section 3, we prove Theorem \ref{conj}, and we give the corresponding results and applications in cluster algebras. In Section 4, we affirm Theorem \ref{abc}.

\section{Preliminaries}

\subsection{Laurent phenomenon algebra}
Let $a,b$ be positive integers satisfying $a\leq b$, write $[a,b]$ for $\{a,a+1,\dots,b\}$.

Let $R$ be a unique factorisation domain over $\mathbb{Z}$, and the\textbf{ ambient field} $\mathcal F$ be the rational function
field in $n$ independent variables over the field of fractions Frac($R$). Recall that an element $f$ of $R$ is \textbf{irreducible} if it is non-zero, not a unit, and not be expressed as the product $f=gh$ of two elements $g,h\in R$ which are non-units.

\begin{definition}
A \textbf{Laurent phenomenon (LP) seed} of rank $n$ in $\mathcal F$ is a pair $(\mathbf{x,F})$, in which

(i) $\mathbf{x}=\{x_1\cdots,x_n\}$ is a transcendence basis for $\mathcal F$ over Frac($R$), where $\mathbf{x}$ is called the {\bf cluster} of  $(\mathbf{x,F})$ and $x_1\cdots,x_n$ are called {\bf cluster variables}.

(ii) $\mathbf{F}=\{F_1,\cdots,F_n\}$  is a collection of irreducible polynomials in $R[x_1,\cdots,x_n]$ such that for each $i,j\in[1,n]$, $x_j\nmid F_i$ ($F_i$ is not divisible by $x_j$) and $F_i$ does not depend on $x_i$, where $F_1,\cdots,F_n$ are called the {\bf exchange polynomials} of $(\mathbf{x,F})$.
\end{definition}

The following notations, definitions and propositions can refer \cite{LP,W}.

Let $F,N$ be two rational functions in $x_1,\cdots,x_n$. Denote by $F|_{x_j\leftarrow N}$ the expression obtained by substituting $x_j$ in $F$ by $N$. And if $F$ involves the variable $x_i$, then we write $x_i\in F$. Otherwise, we write $x_i\notin F$.

\begin{definition}
Let $(\mathbf{x,F})$ be a LP seed in $\mathcal F$. For each $F_j\in{\bf F}$, define a Laurent polynomial $\hat{F}_j=\frac{F_j}{x_1^{a_1}\cdots x_{j-1}^{a_{j-1}}x_{j+1}^{a_{j+1}}\cdots x_n^{a_n}},$ where $a_k\in\mathbb{Z}_{\geq0}$ is maximal such that $F_k^{a_k}$ divides $F_j|_{x_k\leftarrow F_k/x_k'}$ as an element in $R[x_1,\cdots,x_{k-1},(x_k')^{-1},x_{k+1},\cdots,x_n]$. The Laurent polynomials in $ \mathbf{\hat{F}} :=\{\hat{F}_1,\cdots,\hat{F}_n\}$ are called the
{\bf exchange Laurent polynomials}.
\end{definition}

From the definition of exchange Laurent polynomials, we know that $F_j/\hat{F}_j$ is a monomial in $R[x_1,\cdots,\hat{x}_j,\cdots,x_n]$, where $\hat{x}_j$ means $x_j$ vanishes in the $\{x_1,\cdots,x_n\}$. And $\hat{F}_j|_{x_k\leftarrow F_k/x_k'}$ is not divisible by $F_k$.

\begin{proposition}(Lemma 2.4 of \cite{LP})
Let $(\mathbf{x,F})$ be a LP seed in $\mathcal F$, then $\mathbf{F}=\{F_1,\cdots,F_n\}$ and $\mathbf{\hat{F}}=\{\hat{F}_1,\cdots,\hat{F}_n\}$ determine each other uniquely.
\end{proposition}

\begin{proposition}(Lemma 2.7 of \cite{LP})
If $x_k\in F_i$, then $x_i\notin F_k/\hat{F}_k$. In particular, $x_k\in F_i$ implies that $\hat{F}_k|_{x_i\leftarrow0}$ is well defined and $\hat{F}_k|_{x_i\leftarrow0}\in R[x_1^{\pm1},\cdots,\hat x_i,\cdots,\hat x_k,\cdots,x_n^{\pm1}]$.
\label{2.7ofLPA}
\end{proposition}

\begin{definition}
Let $(\mathbf{x,F})$ be a LP seed in $\mathcal F$ and  $k\in[1,n]$. Define a new pair $$(\{x_1',\cdots,x_n'\},\{F_1',\cdots,F_n'\}):=\mu_k(\mathbf{x,F}),$$
where $x_k'=\hat{F}_k/x_k$ and $x_i'=x_i$ for $i \neq k$. And the exchange polynomials change as follows:

 (1) $F_k':=F_k$;

  (2) If $x_k\notin F_i$, then $F_i':=r_iF_i$, where $r_i$ is a unit in $R$;

  (3) If $x_k\in F_i$, then $F_i'$ is obtained from the following three steps:\\
(i) Define $G_i:=F_i|_{x_k\leftarrow N_k}$, where $N_k=\frac{\hat{F}_k|_{x_i\leftarrow0}}{x_k'}$. Then we have
$$G_i\in R[x_1^{\pm1},\cdots,\hat x_i,\cdots, {x_k'}^{-1},\cdots,x_n^{\pm1}]=R[{x_1'}^{\pm1},\cdots,\hat {x_i'},\cdots, {x_k'}^{-1},\cdots,{x_n'}^{\pm1}].$$
(ii) Define $H_i$ to be $G_i$ with all common factors (in $R[x_1,\cdots,\hat x_i,\cdots,\hat x_k,\cdots,x_n]$) with $\hat{F}_k|_{x_i\leftarrow0}$ removed. Note that $H_i$ is unique up to a unit in $R$  and
$H_i\in R[{x_1'}^{\pm1},\cdots,\hat {x_i'},\cdots, {x_k'}^{-1},\cdots,{x_n'}^{\pm1}].$
(iii) Let $M$ be a Laurent monomial in
$x_1',\cdots,\hat{x_i'},\cdots,x_n'$
with coefficient a unit in $R$ such that $F_i':=MH_i\in R[x_1',\cdots,x_n']$ and is not divisible by any variable in $\{x_1',\cdots,x_n'\}$. Thus
$$F_i'\in R[{x_1'},\cdots,\hat {x_i'},\cdots, {x_k'},\cdots,{x_n'}].$$
Then we say that the new pair $\mu_k(\mathbf{x,F})$ is obtained from the LP seed $(\mathbf{x,F})$ by the \textbf{LP mutation} in direction $k$.
\end{definition}
\begin{example}
Let $R=\mathbb{Z}$ and $\mathcal F=\mathbb{Q}(a,b,c)$. Consider the LP seed $(\mathbf{x,F})$, where $\mathbf{x}=\{a,b,c\}$ and $\mathbf F=\{b+1,\ a+c,\ b+1\}$.
From the definition of exchange Laurent polynomials, we can get
$\hat{F}_a=\frac{F_a}{c},~\hat{F}_b=F_b,~\hat{F}_c=\frac{F_c}{a}$.

Let $(\mathbf{x}',\mathbf{F}')=\mu_a(\mathbf{x,F})$, then we have
$a'=\frac{\hat{F_a}}{a}=\frac{b+1}{ac},\ b'=b,\ c'=c$.
From the definition of the LP mutation, the exchange polynomial $F_a$ does not change.
Since $a\notin F_c$, we have $F_c'=b+1$ (or up to a unit).
Since $F_b$ depends on $a$, to compute $F_b'$, we need to procedure the above three steps.
By (i), we get $N_a=\frac{1}{a'c}$ and $G_b=\frac{1}{a'c}+c$.
By (ii), we get $H_b=G_b$ up to a unit in $R$. By (iii), $M=a'c$ and $F_b'=MH_b=a'c^2+1$.
Thus the new seed can be chosen to be
$$(\mathbf{x}',\mathbf{F}')=\{(a',b+1),(b,a'c^2+1),(c,b+1)\}.$$
\label{eg}
\end{example}
\begin{proposition}(Proposition 2.15 of \cite{LP})
Let $(\mathbf{x,F})$ be a LP seed in $\mathcal{F}$, then $\mu_k(\mathbf{x,F})$ is also a LP seed in $\mathcal F$.
\end{proposition}

\begin{proposition}(Proposition 2.16 of \cite{LP})
If $(\mathbf{x}',\mathbf{F}')$ is obtained from $(\mathbf{x,F})$ by LP mutation at $k$, then $(\mathbf{x,F})$ can be obtained from $(\mathbf{x}',\mathbf{F}')$ by LP mutation at $k$. In this sense, LP mutation is an involution.
\end{proposition}

\begin{remark}\label{rmk}
It is important to note that because of (ii), $F_i'$ is defined up to an unit in $R$. And this is the motivation to consider LP seeds up to a equivalent relation.
\end{remark}
\begin{definition}\label{defrelation}
Let $\Sigma_{t_1}=(\mathbf{x}_{t_1},\mathbf{F}_{t_1})$ and $\Sigma_{t_2}=(\mathbf{x}_{t_2},\mathbf{F}_{t_2})$ be two LP seeds in $\mathcal F$. $\Sigma_{t_1}$ and $\Sigma_{t_2}$ are {\bf equivalent} if  for each $i \in [1,n]$,
there exist $r_i,r'_i$ which are units in $R$ such that
$x_{i;t_2}=r_ix_{i;t_1}$ and $F_{i;t_2}=r'_iF_{i;t_1}$. 
\end{definition}

Denote by $[\Sigma_{t}]$ the equivalent class of $\Sigma_{t}$, that is, $[\Sigma_{t}]$ is the set of LP seeds which are equivalent to $\Sigma_{t}$.

It is not a clear priori that the LP mutation  $\mu_k(\mathbf{x,F})$  of a LP seed $(\mathbf{x,F})$
is still a LP seed because of the irreducibility requirement for the new exchange polynomials. But it can be seen from the following proposition that $\mu_k(\mathbf{x,F})$ is still a LP seed in $\mathcal F$.

\begin{proposition}(Lemma 3.1 of \cite{LP})
Let  $\Sigma_{t_1}=(\mathbf{x}_{t_1},\mathbf{F}_{t_1})$ and $\Sigma_{t_2}=(\mathbf{x}_{t_2},\mathbf{F}_{t_2})$ be two LP seeds in $\mathcal F$, and $\Sigma_{t_u}=\mu_k(\Sigma_{t_2})~$, $\Sigma_{t_v}=\mu_{k}(\Sigma_{t_1})$.
If $[\Sigma_{t_1}]=[\Sigma_{t_2}]$, then $[\Sigma_{t_v}]=[\Sigma_{t_u}]$.
\end{proposition}

Let $\Sigma_{t}=(\mathbf{x}_t,\mathbf{F}_t)$ be a LP seed in $\mathcal F$. By the above proposition, it is reasonable to define LP mutation of $[\Sigma_{t}]$ at $k$ given by
$\mu_k([\Sigma_{t}]):=[\mu_k(\Sigma_{t})]$.

\begin{definition}
A {\bf Laurent phenomenon (LP) pattern $\mathcal{S}$} in $\mathcal F$ is an assignment for each  LP seed  $(\mathbf{x}_t,\mathbf{F}_t)$ to a vertex $t$ of the $n$-regular tree $\mathbb T_n$, such that for any edge $t^{~\underline{\quad k \quad}}~ t',~(\mathbf{x}_{t'},\mathbf{F}_{t'})=\mu_k(\mathbf{x}_t,\mathbf{F}_t)$.
\end{definition}

We always denote by
$\mathbf{x}_t=\{x_{1;t},\cdots,x_{n;t}\}$ and $\mathbf{F}_t=\{F_{1;t},\cdots,F_{n;t}\}$.

\begin{definition}
Let $\mathcal{S}$ be a LP pattern, the {\bf Laurent phenomenon (LP) algebra $\mathcal{A}(\mathcal{S})$} (of rank $n$) associated with  $\mathcal{S}$ is the $R$-subalgebra of $\mathcal F$ generated by all the cluster variables in the seeds of $\mathcal{S}$.
\end{definition}

If $\Sigma=(\mathbf{x,F})$ is any seed in $\mathcal{F}$, we shall write $\mathcal{A}(\Sigma)$ to mean the LP algebra $\mathcal{A}(\mathcal{S})$ associated with $\mathcal{S}$ containing the seed $\Sigma$ .

\begin{theorem}(Theorem 5.1 of \cite{LP}, Laurent phenomenon)
\label{thmlaurent}
Let $\mathcal{A}(\mathcal{S})$ be a LP algebra, and $(\mathbf{x}_{t_0},\mathbf{F}_{t_0})$ be a LP seed of $\mathcal{A}(\mathcal{S})$. Then any cluster variable $x_{i;t}$ of $\mathcal{A}(\mathcal{S})$ is in the Laurent polynomial ring $R(t_0^{\pm1}):=R[x_{1;t_0}^{\pm1},\cdots,x_{n;t_0}^{\pm1}]$.
\end{theorem}

\begin{definition}
Let $\Sigma=(\mathbf{x,F})$ be a LP seed of rank $n$ and $k\in[1,n]$. A new seed $\Sigma^\ast=(\mathbf{x^\ast,F^\ast})$ of rank $n-1$ is defined as follows:
\begin{enumerate}
  \item let $R^\ast=R[x_k^{\pm1}]$.
  \item $\mathbf{x^\ast}=\mathbf{x}-\{x_k\}$.
  \item let $\mathbf{F^\ast}=\{F_j^\ast|j\in [1,n]-k,\ F_j^\ast=F_jx_k^a,\ where\ a\ is\ the\ power\ of\ x_k\ in\ \hat{F}_k\}$.
\end{enumerate}
The seed $\Sigma^\ast$ is in fact a LP seed, then $\Sigma^\ast$ is called {\bf the freezing of the LP seed }$\Sigma$ at $x_k$.
$\mathcal{A}(\Sigma^\ast)\subset \mathcal{F}=\text{Frac}(R^\ast[x_1,\dots,\hat{x_k},\dots,x_n])$ is defined to be the subalgebra generated by all the cluster variables from LP seeds mutation-equivalent to $\Sigma^\ast$.
Then $\mathcal{A}(\Sigma^\ast)$ is called the {\bf freezing of the LP algebra }$\mathcal{A}(\Sigma)$ at $x_k$.
\end{definition}

\begin{example}
Consider the LP seed $\Sigma=(\mathbf{x,F})=\{(a,b+1),(b,a+c),(c,b+1)\}$ over $R=\mathbb{Z}$ from Example \ref{eg}.
We produce the freezing of $(\mathbf{x,F})$ at $c$ as follows:
first, remove $(c,b+1)$;
next, since the powers of $c$ in $\hat{F_a}$ and $\hat{F_b}$ are -1 and 0 respectively, we have $F_a^\ast=F_ac^{-1}=\frac{b+1}{c}$ and $F_b^\ast=F_b$.
Then the LP seed $\Sigma^\ast$ are $\{(a,\frac{b+1}{c}),(b,a+c)\}$ over $\mathbb{Z}[c^{\pm1}]$.
\end{example}

\begin{proposition}(Lemma 3.1 of \cite{LP})
The algebra $\mathcal{A}(\Sigma^\ast)$ is a LP algebra.
\end{proposition}

\begin{corollary}
The freezing of the LP seed at $x_i$ is compatible with the mutation in direction $j$ for $j\neq i$.
\end{corollary}

\subsection{Cluster algebra}
Recall that an integer matrix $B_{n\times n}=(b_{ij})$ is  called \textbf{skew-symmetrizable} if there is a positive integer diagonal matrix $D$ such that $DB$ is skew-symmetric, where $D$ is said to be the  \textbf{skew-symmetrizer} of $B$.
$B_{n\times n}=(b_{ij})$ is {\bf sign-skew-symmetric} if $b_{ij}b_{ji}<0$ or $b_{ij}=b_{ji}=0$ for any $i,j\in [1,n]$.
A sign-skew-symmetric $B$ is \textbf{totally sign-skew-symmetric} if any matrix $B'$ obtained from $B$ by a sequence of mutations is sign-skew-symmetric.
It is known that skew-symmetrizable integer matrices are always totally sign-skew-symmetric.

The \textbf{diagram} for a sign-skew-symmetric matrix $B_{n\times n}$ is the directed graph $\Gamma(B)$ with the vertices $1,2,\cdots,n$ and the directed edges from $i$ to $j$ if $b_{ij}>0$. $B_{n\times n}$ is called {\bf acyclic} if $\Gamma(B)$ has no oriented cycles. As shown in \cite{HL}, an acyclic sign-skew-symmetric integer matrix $B$ is always totally sign-skew-symmetric.

Let $\mathbb{P}$ be the coefficient group, its group ring $\mathbb{Z}\mathbb{P}$ is a domain \cite{FZ}.
We take an ambient field $\mathcal F$  to be the field of rational functions in $n$ independent variables with coefficients in $\mathbb{Z}\mathbb{P}$.

\begin{definition}
A \textbf{cluster seed} in $\mathcal F$ is a triplet $\Sigma=(\mathbf{x},\mathbf{y},B)$ such that

(i)  $\mathbf{x}=\{x_1,\cdots, x_n\}$ is a transcendence basis for $\mathcal F$ over Frac($\mathbb {ZP}$). $\mathbf{x}$ is called the  {\bf cluster} of  $(\mathbf{x},\mathbf{y},B)$ and $\{x_1\cdots,x_n\}$  are called {\bf cluster variables}.

(ii) $\mathbf{y}=\{y_1,\cdots,y_n\}$ is a subset of  $\mathbb{P}$, where $\{y_1,\cdots,y_n\}$ are called {\bf coefficients}.

(iii) $B=(b_{ij})$ is a $n\times n$ totally sign-skew-symmetric matrix, called an {\bf exchange matrix}.
\end{definition}

Let $(\mathbf{x},\mathbf{y}, B)$ be a cluster seed in $\mathcal F$, one can associate binomials $\{F_1,\cdots,F_n\}$ defined by
\begin{eqnarray}
F_j=\frac{y_j}{1\oplus y_j}\prod\limits_{b_{ij}>0}x_i^{b_{ij}}+\frac{y_j}{1\oplus y_j}\prod\limits_{b_{ij}<0}x_i^{-b_{ij}}.\nonumber
\end{eqnarray}
$\{F_1,\cdots,F_n\}$ are called the {\bf exchange polynomials} of $(\mathbf{x},\mathbf{y}, B)$.

Note that the coefficients and the exchange matrices in a cluster algebra are used for providing the exchange polynomials and explaining how to produce new exchange polynomials when doing a mutation (see Definition \ref{defmutation}) on a cluster seed.


\begin{definition}\label{defmutation}
Let $\Sigma=(\mathbf{x},\mathbf{y},B)$ be a cluster seed in $\mathcal F$. Define the {\bf mutation}  of  $\Sigma$ in the direction $k\in[1,n]$ as a new triple $\Sigma'=(\mathbf{x}', \mathbf{y}',B'):=\mu_k(\Sigma)$ in $\mathcal F$, where
$$x_i'=\begin{cases} F_k/x_k~& i=k\\x_i~& i\neq k.\end{cases},~
 y_i'=\begin{cases} y_k^{-1}~& i=k\\ y_iy_k^{max(b_{ki},0)}(1\bigoplus y_k)^{-b_{ki}}~& i\neq k.\end{cases},$$
 $$\text{ and } b_{ij}'=\begin{cases}-b_{ij}~& i=k\text{ or } j=k\\ b_{ij}+sgn(b_{ik})max(b_{ik}b_{kj},0)~&otherwise\end{cases}.$$
\end{definition}

It can be seen that $\mu_k(\Sigma)$ is also a cluster seed and mutation of a cluster seed is an involution, that is, $\mu_k(\mu_k(\Sigma))=\Sigma$.

\begin{definition}
A {\bf cluster pattern} $\mathcal{S}$ is an assignment of a seed  $\Sigma_t=(\mathbf{x}_t,\mathbf{y}_t,B_t)$ to every vertex $t$ of the $n$-regular tree $\mathbb T_n$, such that for any edge $t^{~\underline{\quad k \quad}}~ t',~\Sigma'_{t'}=(\mathbf{x}_{t'},\mathbf{y}_{t'},B_{t'})=\mu_k(\Sigma_t)$.
\end{definition}
We always denote by $\mathbf{x}_t=(x_{1;t},\cdots, x_{n;t}),~ \mathbf{y}_t=(y_{1;t},\cdots, y_{n;t}), ~B_t=(b_{ij}^t).$
\begin{definition} Let  $\mathcal{S}$ be a cluster pattern,   the {\bf cluster algebra} $\mathcal{A}(\mathcal{S})$ (of rank $n$) associated with the given cluster pattern $\mathcal{S}$ is the $\mathbb {ZP}$-subalgebra of the field $\mathcal F$ generated by all cluster variables of  $\mathcal{S}$.
 \end{definition}

\begin{theorem}(Theorem 3.1 of \cite{FZ}, Laurent phenomenon)
Let $\mathcal{A}(\mathcal{S})$ be a cluster algebra, and $\Sigma_{t_0}=(\mathbf{x}_{t_0},\mathbf{y}_{t_0},B_{t_0})$ be a cluster seed of $\mathcal{A}(\mathcal{S})$. Then any cluster variable $x_{i;t}$ of $\mathcal{A}(\mathcal{S})$ is in the Laurent polynomial ring $\mathbb {ZP}(t_0^{\pm1}):=\mathbb {ZP}[x_{1;t_0}^{\pm1},\cdots,x_{n;t_0}^{\pm1}]$.
\end{theorem}

 \begin{example}\label{expoly}
Let $B=\begin{pmatrix}0&3\\-3&0\end{pmatrix}$, then the exchange polynomials of the  cluster seed $(\mathbf{x}, B)$ are the following two polynomials
$$F_1=x_2^3+1=(x_2+1)(x_2^2-x_2+1),$$
$$F_2=x_1^3+1=(x_1+1)(x_1^2-x_1+1).$$
\end{example}
It is easy to see that the exchange polynomials $F_1, F_2$ of $(\mathbf{x},B)$ are both reducible in the above example. Thus
the cluster $\mathbf{x}$ and the exchange polynomial $\mathbf{F}$ of $(\mathbf{x},B)$ can not define a LP seed.
From \cite{LP}, we know that sometimes a cluster algebra defines a LP algebra indeed.

\begin{theorem}(Theorem 4.5 of \cite{LP})
  Every cluster algebra with principal coefficients is a Laurent phenomenon algebra.
\end{theorem}

\section{LP seeds determined by either clusters or mutations}
\subsection{On Theorem \ref{conj} (1)}

\begin{theorem} \label{mainthm1}
Let $\mathcal{A}(\mathcal{S})$ be a LP algebra of rank $n$, and $(\mathbf{x}_{t_1},\mathbf{F}_{t_1}), (\mathbf{x}_{t_2},\mathbf{F}_{t_2})$ be two LP seeds of $\mathcal{A}(\mathcal{S})$.
\begin{enumerate}
  \item If there exists a permutation $\sigma$ of $[1,n]$ and an unit $r_i\in R$ such that $x_{i;t_2}=r_ix_{\sigma(i);t_1}$ for $i\in [1,n]$,
then $F_{i;t_2}=r'_iF_{\sigma(i);t_1}$ as polynomials for a certain unit $r'_i$ in $R$.
  \item each LP seed is uniquely defined by its cluster.
\end{enumerate}
\end{theorem}
\begin{proof}
Without loss generality, we assume that $r_1=\cdots=r_n=1$. It does not make difference to the proof.

For any fixed $k\in[1,n]$, let $(\mathbf{x}_{u},{\bf F}_{u})=\mu_k(\mathbf{x}_{t_2},\mathbf{F}_{t_2})$ and $(\mathbf{x}_{v},{\bf F}_{v})=\mu_{\sigma(k)}(\mathbf{x}_{t_1},\mathbf{F}_{t_1})$, we consider the Laurent expansion of $x_{k;u}$ with respect to $\mathbf{x}_v$ and the  Laurent expansion of $x_{\sigma(k);v}$ with respect to $\mathbf{x}_u$.

From the definition of the LP mutation, we know
\begin{eqnarray}
x_{i;u}=\begin{cases}x_{i;t_2}&\text{if }i\neq k\\ \frac{\hat{F}_{k;t_2}}{x_{k;t_2}} &\text{if } i=k\end{cases}\text{  and \ }
x_{i;v}=\begin{cases}x_{i;t_1}&\text{if }i\neq \sigma(k)\\ \frac{\hat{F}_{\sigma(k);t_1}}{x_{\sigma(k);t_1}} &\text{if } i=\sigma(k)\end{cases}.
\label{eqmu}
\end{eqnarray}
By $x_{i;t_2}=x_{\sigma(i);t_1}$ for $i\in [1,n]$, we have $x_{i;u}=x_{\sigma(i);v}$ for $i\neq k$. By (\ref{eqmu}), we get
\begin{eqnarray}
x_{k;u}&=&\hat{F}_{k;t_2}(x_{1;t_2},\cdots,\hat x_{k;t_2},\cdots,x_{n;t_2})/x_{k;t_2}\nonumber\\
&=&\hat{F}_{k;t_2}(x_{\sigma(1);t_1},\cdots,\hat x_{\sigma(k);t_1},\cdots,x_{\sigma(n);t_1})/x_{\sigma(k);t_1}\nonumber\\
&=&\hat{F}_{k;t_2}(x_{\sigma(1);v},\cdots,\hat x_{\sigma(k);v},\cdots,x_{\sigma(n);v})/x_{\sigma(k);t_1};\nonumber\\
x_{\sigma(k);v}&=&\hat{F}_{\sigma(k);t_1}(x_{1;t_1},\cdots,\hat x_{\sigma(k);t_1},\cdots,x_{n;t_1})/x_{\sigma(k);t_1}\nonumber\\
&=&\hat{F}_{\sigma(k);t_1}(x_{1;v},\cdots,\hat x_{\sigma(k);v},\cdots,x_{n;v})/x_{\sigma(k);t_1}.\nonumber
\end{eqnarray}
Thus $\frac{x_{k;u}}{x_{\sigma(k);v}}=\frac{\hat{F}_{k;t_2}(x_{\sigma(1);v},\cdots,\hat x_{\sigma(k);v},\cdots,x_{\sigma(n);v})}{\hat{F}_{\sigma(k);t_1}(x_{1;v},\cdots,\hat x_{\sigma(k);v},\cdots,x_{n;v})}$  and we get that
\begin{eqnarray}\label{eqlaurent}
x_{k;u}= x_{\sigma(k);v}\frac{\hat{F}_{k;t_2}(x_{\sigma(1);v},\cdots,\hat x_{\sigma(k);v},\cdots,x_{\sigma(n);v})}{\hat{F}_{\sigma(k);t_1}(x_{1;v},\cdots,\hat x_{\sigma(k);v},\cdots,x_{n;v})}.
\end{eqnarray}

From the definition of the exchange Laurent  polynomial, we know the above equation has the form of
\begin{eqnarray}\label{eqexpansion1}
x_{k;u}=x_{\sigma(k);v}\frac{F_{k;t_2}(x_{\sigma(1);v},\cdots,\hat x_{\sigma(k);v},\cdots,x_{\sigma(n);v})}{F_{\sigma(k);t_1}(x_{1;v},\cdots,\hat x_{\sigma(k);v},\cdots,x_{n;v})}M,
\end{eqnarray}
where the Laurent monomial $M$ is of the form $x_{1;v}^{m_1}\cdots x_{\sigma(k)-1;v}^{m_{\sigma(k)-1}}x_{\sigma(k)+1;v}^{m_{\sigma(k)+1}}\cdots x_{n;v}^{m_n}$ and $m_j$ is integer for $j\in [1,n]-\sigma(k)$. 
Thus equation (\ref{eqexpansion1}) is the Laurent expansion of $x_{k;u}$ with respect to $\mathbf{x}_v$.

Similarly, the following equation is the Laurent expansion of $x_{\sigma(k);v}$ with respect to $\mathbf{x}_u$.
\begin{eqnarray}\label{eqexpansion2}
x_{\sigma(k);v}=x_{k;u}\frac{F_{\sigma(k);t_1}(x_{\sigma^{-1}(1);u},\cdots,\hat x_{k;u},\cdots,x_{\sigma^{-1}(n);u})}{F_{k;t_2}(x_{1;u},\cdots,\hat x_{k;u},\cdots,x_{n;u})}M^{-1},
\end{eqnarray}
where $M^{-1}$ is also a Laurent monomial in $R[x_{1;u}^{\pm1},\cdots,\hat x_{k;u},\cdots,x_{n;u}^{\pm1}]$ since $x_{i;u}=x_{\sigma(i);v}$ for $i\neq k$.
\begin{spacing}{1.7}
We know that both
$\frac{F_{\sigma(k);t_1}(x_{\sigma^{-1}(1);u},\cdots,\hat x_{k;u},\cdots,x_{\sigma^{-1}(n);u})}{F_{k;t_2}(x_{1;u},\cdots,\hat x_{k;u},\cdots,x_{n;u})}=\frac{ F_{\sigma(k);t_1}(x_{1;v},\cdots,\hat x_{\sigma(k);v},\cdots,x_{n;v})}{ F_{k;t_2}(x_{\sigma(1);v},\cdots,\hat x_{\sigma(k);v},\cdots,x_{\sigma(n);v})}$
and
$\frac{ F_{k;t_2}(x_{\sigma(1);v},\cdots,\hat x_{\sigma(k);v},\cdots,x_{\sigma(n);v})}{ F_{\sigma(k);t_1}(x_{1;v},\cdots,\hat x_{\sigma(k);v},\cdots,x_{n;v})}$
are Laurent polynomials in $$R[x_{1;v}^{\pm 1},\cdots,\hat x_{\sigma(k);v},\cdots,x_{n;v}^{\pm 1}]=R[x_{1;u}^{\pm 1},\cdots,\hat x_{k;u},\cdots,x_{n;u}^{\pm 1}]$$ by Laurent phenomenon.

Thus both $\frac{ F_{k;t_2}(x_{\sigma(1);v},\cdots,\hat x_{\sigma(k);v},\cdots,x_{\sigma(n);v})}{ F_{\sigma(k);t_1}(x_{1;v},\cdots,\hat x_{\sigma(k);v},\cdots,x_{n;v})}$ and $\frac{ F_{\sigma(k);t_1}(x_{1;v},\cdots,\hat x_{\sigma(k);v},\cdots,x_{n;v})}{ F_{k;t_2}(x_{\sigma(1);v},\cdots,\hat x_{\sigma(k);v},\cdots,x_{\sigma(n);v})}$ are units in \\ $R[x_{1;v}^{\pm 1},\cdots,\hat x_{\sigma(k);v},\cdots,x_{n;v}^{\pm 1}].$

Because both $F_{k;t_2}$ and $F_{\sigma(k);t_1}$ are irreducible and
$x_{j;t_2} \nmid F_{k;t_2}$, $x_{j;t_1}\nmid F_{\sigma(k);t_1}$ for each $j\in[1,n]$, so that both $\frac{ F_{k;t_2}(x_{\sigma(1);v},\cdots,\hat x_{\sigma(k);v},\cdots,x_{\sigma(n);v})}{ F_{\sigma(k);t_1}(x_{1;v},\cdots,\hat x_{\sigma(k);v},\cdots,x_{n;v})}$ and $\frac{ F_{\sigma(k);t_1}(x_{1;v},\cdots,\hat x_{\sigma(k);v},\cdots,x_{n;v})}{ F_{k;t_2}(x_{\sigma(1);v},\cdots,\hat x_{\sigma(k);v},\cdots,x_{\sigma(n);v})}$ are units in $R$.
\end{spacing}
Hence
$$F_{k;t_2}(x_{\sigma(1);v},\cdots,\hat x_{\sigma(k);v},\cdots,x_{\sigma(n);v})=r_k'F_{\sigma(k);t_1}(x_{1;v},\cdots,\hat x_{\sigma(k);v},\cdots,x_{n;v}),$$ for some unit $r_k'$ in $R$,
i.e., $F_{k;t_2}(x_{1;u},\cdots,\hat x_{k;u},\cdots,x_{n;u})=r_k'F_{\sigma(k);t_1}(x_{1;v},\cdots,\hat x_{\sigma(k);v},\cdots,x_{n;v})$. Thus
$F_{k;t_2}=r_k'F_{\sigma(k);t_1}$ as polynomials, for $k\in[1,n]$.
\end{proof}

In fact, the proof of the above theorem also works for any (totally sign-skew-symmetric) cluster algebra with any coefficients and any generalized cluster algebra. We do not talk much about generalized cluster algebras here, and one can refer \cite{CL,CS,NT} for details.
Now we give the result for cluster algebras and main points of proof that are different from the previous one.
\begin{theorem} \label{mainthm3}
Let $\mathcal{A}(\mathcal{S})$ be a cluster algebra, and $\Sigma_{t_l}=(\mathbf{x}_{t_l},\mathbf{y}_{t_l},B_{t_l}),\; l=1,2$ be two cluster seeds of $\mathcal{A}(\mathcal{S})$.
If  there exists a permutation $\sigma$ of  $[1,n]$  such that $x_{i;t_2}=x_{\sigma(i);t_1}$ for $i\in[1,n]$,
then

(i) Either $y_{k;t_2}=y_{\sigma(k);t_1}$, $~b_{ik}^{t_2}=b_{\sigma(i)\sigma(k)}^{t_1}$ or $y_{k;t_2}=y_{\sigma(k);t_1}^{-1}$, $~b_{ik}^{t_2}=-b_{\sigma(i)\sigma(k)}^{t_1}$ for $i,k\in[1,n]$.

(ii) In both cases, $F_{i;t_2}=F_{\sigma(i);t_1}$ as polynomials for $i\in[1,n]$.
\end{theorem}
\begin{proof}
By the same method with the proof of Theorem \ref{mainthm1}, the version of the equation (\ref{eqlaurent}) for the cluster algebra is just
\begin{eqnarray}\label{eqpan1}
x_{k;u}=x_{\sigma(k);v}\frac{F_{k;t_2}(x_{\sigma(1);v},\cdots,\hat x_{\sigma(k);v},\cdots,x_{\sigma(n);v})}{F_{\sigma(k);t_1}(x_{1;v},\cdots,\hat x_{\sigma(k);v},\cdots,x_{n;v})},
\end{eqnarray}
and note that $x_{i;u}=x_{\sigma(i);v}$ for any $i\neq k$, we also have
\begin{eqnarray}\label{eqpan2}
x_{\sigma(k);v}=x_{k;u}\frac{F_{\sigma(k);t_1}(x_{\sigma^{-1}(1);u},\cdots,\hat x_{k;u},\cdots,x_{\sigma^{-1}(n);u})}{F_{k;t_2}(x_{1;u},\cdots,\hat x_{k;u},\cdots,x_{n;u})}.
\end{eqnarray}

We know that equation (\ref{eqpan1}) is the Laurent expansion of $x_{k;u}$ with respect to $\mathbf{x}_{v}$ and
equation (\ref{eqpan2}) is the Laurent expansion of $x_{\sigma(k);v}$ with respect to $\mathbf{x}_{u}$. Then by Laurent phenomenon, both $\frac{F_{k;t_2}(x_{\sigma(1);v},\cdots,\hat x_{\sigma(k);v},\cdots,x_{\sigma(n);v})}{F_{\sigma(k);t_1}(x_{1;v},\cdots,\hat x_{\sigma(k);v},\cdots,x_{n;v})}$ and $\frac{F_{\sigma(k);t_1}(x_{\sigma^{-1}(1);u},\cdots,\hat x_{k;u},\cdots,x_{\sigma^{-1}(n);u})}{F_{k;t_2}(x_{1;u},\cdots,\hat x_{k;u},\cdots,x_{n;u})}$ are Laurent polynomials, and this implies that $\frac{F_{k;t_2}(x_{\sigma(1);v},\cdots,\hat x_{\sigma(k);v},\cdots,x_{\sigma(n);v})}{F_{\sigma(k);t_1}(x_{1;v},\cdots,\hat x_{\sigma(k);v},\cdots,x_{n;v})}$ is a Laurent monomial in $\mathbb{ZP}[x_{1;v}^{\pm1},\cdots,\hat x_{\sigma(k);v},\cdots,x_{n;v}^{\pm1}]$. We know that
$$F_{k;t_2}(x_{\sigma(1);v},\cdots,\hat x_{\sigma(k);v},\cdots,x_{\sigma(n);v})=\frac{y_{k;t_2}}{1\oplus y_{k;t_2}}\prod\limits_{b_{ik}^{t_2}>0}x_{\sigma(i);v}^{b_{ik}^{t_2}}
+\frac{1}{1\oplus y_{k;t_2}}\prod\limits_{b_{ik}^{t_2}<0}x_{\sigma(i);v}^{-b_{ik}^{t_2}},$$
$$F_{\sigma(k);t_1}(x_{1;v},\cdots,\hat x_{\sigma(k);v},\cdots,x_{n;v})=\frac{y_{\sigma(k);t_1}}{1\oplus y_{\sigma(k);t_1}}\prod\limits_{b_{i\sigma(k)}^{t_1}>0}x_{i;v}^{b_{i\sigma(k)}^{t_1}}
+\frac{1}{1\oplus y_{\sigma(k);t_1}}\prod\limits_{b_{i\sigma(k)}^{t_1}<0}x_{i;v}^{-b_{i\sigma(k)}^{t_1}}.$$
Because $\frac{F_{k;t_2}(x_{\sigma(1);v},\cdots,\hat x_{\sigma(k);v},\cdots,x_{\sigma(n);v})}{F_{\sigma(k);t_1}(x_{1;v},\cdots,\hat x_{\sigma(k);v},\cdots,x_{n;v})}$   is a Laurent monomial, we must have either $y_{k;t_2}=y_{\sigma(k);t_1}$, $~b_{ik}^{t_2}=b_{\sigma(i)\sigma(k)}^{t_1}$ or $y_{k;t_2}=y_{\sigma(k);t_1}^{-1}$,  $~b_{ik}^{t_2}=-b_{\sigma(i)\sigma(k)}^{t_1}$.
In both cases, we have $$F_{k;t_2}(x_{\sigma(1);v},\cdots,\hat x_{\sigma(k);v},\cdots,x_{\sigma(n);v})=F_{\sigma(k);t_1}(x_{1;v},\cdots,\hat x_{\sigma(k);v},\cdots,x_{n;v}),$$
i.e., $F_{k;t_2}(x_{1;u},\cdots,\hat x_{k;u},\cdots,x_{n;u})=F_{\sigma(k);t_1}(x_{1;v},\cdots,\hat x_{\sigma(k);v},\cdots,x_{n;v})$. Thus
$F_{k;t_2}=F_{\sigma(k);t_1}$ as polynomials.
\end{proof}
\begin{lemma}\label{lem3}
Let $\mathcal{A}(\mathcal{S})$ be a skew-symmetrizable cluster algebra with skew-symmetrizer $D$, and
$(\mathbf{x}_{t_1},\mathbf{y}_{t_1},B_{t_1}), (\mathbf{x}_{t_2},\mathbf{y}_{t_2},B_{t_2})$ be two cluster seeds of $\mathcal{A}(\mathcal{S})$.
If  there exists a permutation $\sigma$ of $[1,n]$  such that $x_{i;t_2}=x_{\sigma(i);t_1}$ for $i\in[1,n]$,
then $b_{ik}^{t_2}=\frac{d_k}{d_{\sigma (k)}}b_{\sigma(i)\sigma(k)}^{t_1}$.
\end{lemma}
\begin{proof}
Let $P_\sigma$ be the permutation matrix define by the permutation $\sigma$. By the cluster formula (see Theorem 3.5 of \cite{CL}), we have
$P_\sigma (B_{t_1}D^{-1})P_\sigma^{\top}=B_{t_2}D^{-1}$. Then $B_{t_2}=(P_\sigma B_{t_1}P_\sigma^{\top})( P_\sigma D^{-1}P_\sigma^{\top})D$.
The result follows.

By the proof of the first statement and the definition of equivalence for two cluster seeds, we conclude the second statement.
\end{proof}

From Theorem \ref{mainthm3} and Lemma \ref{lem3}, we can affirm a conjecture for skew-symmetrizable cluster algebra proposed by Fomin and Zelevinsky in \cite{FZ2}, which says every seed of a cluster algebra is uniquely determined by its cluster.
\begin{corollary}
Let $\mathcal{A}(\mathcal{S})$ be a skew-symmetrizable cluster algebra with skew-symmetrizer $D$, and $(\mathbf{x}_{t_1},\mathbf{y}_{t_1},B_{t_1}), (\mathbf{x}_{t_2},\mathbf{y}_{t_2},B_{t_2})$ be two cluster seeds of $\mathcal{A}(\mathcal{S})$.
If  there exists a permutation $\sigma$ of  $[1,n]$  such that $x_{i;t_2}=x_{\sigma(i);t_1}$ for $i\in[1,n]$,
then  $y_{k;t_2}=y_{\sigma(k);t_1}$, $~b_{ik}^{t_2}=b_{\sigma(i)\sigma(k)}^{t_1}$, $d_k=d_{\sigma(k)}$ for any $i$ and $k$.
\end{corollary}

\subsection{On Theorem \ref{conj} (2)}
Let $\mathcal{A}(\mathcal{S})$ be a LP algebra, if there is a seed $(\mathbf{x}_{t_0},{\bf F}_{t_0})$ of $\mathcal{A}(\mathcal{S})$ such that the exchange polynomials in ${\bf F}_{t_0}$ are all nontrivial, we say that  $\mathcal{A}(\mathcal{S})$ is a LP algebra having no trivial exchange relations.

Note that if there is a trivial exchange polynomial in a LP seed $(\mathbf{x}_{t_0},{\bf F}_{t_0})$, from the definition of LP mutation, this trivial exchange polynomial remain invariant under any sequence of LP mutations. So if $\mathcal{A}(\mathcal{S})$ is a LP algebra having no trivial exchange relations, then each exchange polynomial of $\mathcal{A}(\mathcal{S})$ is a nontrivial polynomial.

\begin{lemma}\label{mainlem}
Let $\mathcal{A}(\mathcal{S})$ be a LP algebra having no trivial exchange relations, and $\Sigma_{t}=(\mathbf{x}_t,\mathbf{F}_t)$,
$\Sigma_{t_0}=(\mathbf{x}_{t_0},\mathbf{F}_{t_0})$ be two LP seeds of $\mathcal{A}(\mathcal{S})$ with $x_{i;t}=r_ix_{i;t_0}$, where $r_i$ is an unit in $R$ for any $i\neq k$. If $x_{k;t}=Mx_{k;t_0}$ for some Laurent monomial $M$ in $R[x_{1;t_0}^{\pm1},\cdots,\hat x_{k;t_0},\cdots,x_{n;t_0}^{\pm1}]$, then $M$ is an unit in $R$, and $[\Sigma_{t}]=[\Sigma_{t_0}]$.
\end{lemma}
\begin{proof}
Without loss generality, we assume that $r_i=1$ for $i\neq k$. It does not make difference to the proof.

Assume that $M=r\prod\limits_{i\neq k}x_{i;t_0}^{a_i}=r\prod\limits_{i\neq k}x_{i;t}^{a_i}$, where $r$ is an unit in $R$.
If there exists some $j\neq k$ such that $a_j<0$, then we consider the LP seed $(\mathbf{x}_{w}, \mathbf{F}_{w})=\mu_j(\Sigma_{t_0})$. From the definition of LP mutation, we know that $x_{i;w}=x_{i;t_0}$ for $i\neq j$ and $x_{j;w}x_{j;t_0}=\hat{F}_{j;t_0}(x_{1;t_0},\cdots,\hat x_{j;t_0},\cdots,x_{n;t_0})$. Then we have $$x_{k;t}=(r\prod\limits_{i\neq k}x_{i;t_0}^{a_i})x_{k;t_0}=\frac{(r\prod\limits_{i\neq j,k}x_{i;w}^{a_i})x_{j;w}^{-a_j}}{\hat{F}_{j;t_0}^{-a_j}(x_{1;w},\cdots,\hat x_{j;w},\cdots,x_{n;w})}x_{k;w},$$
which can be written as the following equation, from the definition of the exchange Laurent polynomial.
\begin{eqnarray}\label{eqlem}
x_{k;t}=\frac{(r\prod\limits_{i\neq j,k}x_{i;w}^{a_i})x_{j;w}^{-a_j}L}{F_{j;t_0}^{-a_j}(x_{1;w},\cdots,\hat x_{j;w},\cdots,x_{n;w})}x_{k;w},
\end{eqnarray}
where $L$ is a Laurent monomial in $R[x_{1;w}^{\pm1},\cdots,\hat x_{j;w},\cdots,x_{n;w}^{\pm1}]$.
Thus equation (\ref{eqlem}) is the expansion of $x_{k;t}$ with respect to $\mathbf{x}_w$.
Because $\mathcal{A}(\mathcal{S})$  has no trivial exchange relations, $F_{j;t_0}$ is a nontrivial polynomial. And we know that $F_{j;t_0}$ is irreducible and $x_s\nmid F_{j;t_0}$ for each $s\in [1,n]$, thus equation (\ref{eqlem}) will contradict Laurent phenomenon.
So each $a_j$ is nonnegative.

Similarly, by considering that $x_{k;t_0}=M^{-1}x_{k;t}=(r\prod\limits_{i\neq k}x_{i;t}^{-a_i})x_{k;t}$, we can get each $-a_j$ is nonnegative.
Thus each $a_j$ is $0$, thus $M=r$ is an unit in $R$. Then by Theorem \ref{mainthm1}, we have $[\Sigma_{t}]=[\Sigma_{t_0}]$.
\end{proof}

\begin{theorem}\label{mainthm2}
Let $\mathcal{A}(\mathcal{S})$ be a LP algebra of rank $n$ having no trivial exchange relations, and $\Sigma_{t_1}=(\mathbf{x}_{t_1},\mathbf{F}_{t_1})$,$\Sigma_{t_2}=(\mathbf{x}_{t_2},\mathbf{F}_{t_2})$ be two LP seeds of $\mathcal{A}(\mathcal{S})$. If $x_{i;t_1}=r_ix_{i;t_2}$ holds for any $i\neq k$, where $r_i$ is an unit in $R$, then $[\Sigma_{t_1}]=[\Sigma_{t_2}]$ or $[\Sigma_{t_1}]=\mu_k[\Sigma_{t_2}]$, that is, any two LP seeds with $n>1$ common cluster variables are connected with each other by one step of mutation.
\end{theorem}
\begin{proof}
Without loss generality, we assume that $r_i=1$ for $i\neq k$. It does not make difference to the proof.

By Laurent phenomenon, we assume that $x_{k;t_2}=f(x_{1;t_1},\cdots,x_{n;t_1})$ and $x_{k;t_1}=g(x_{1;t_2},\cdots,x_{n;t_2})$, where $f\in R[x_{1;t_1}^{\pm1},\cdots,x_{n;t_1}^{\pm1}]$ and $g\in R[x_{1;t_2}^{\pm1},\cdots,x_{n;t_2}^{\pm1}]$. Since $x_{i;t_1}=x_{i;t_2}$ for any $i\neq k$, we know that $x_{k;t_1}$ entries $f$ with exponent $1$ or $-1$; Thus $x_{k;t_2}$ has the form of $x_{k;t_2}=L_1x_{k;t_1}^{\pm1}+L_0$, where  $L_1\neq 0$   and $L_0$ are Laurent polynomials in $$R[x_{1;t_1}^{\pm1},\cdots,\hat x_{k;t_1},\cdots,x_{n;t_1}^{\pm1}]=R[x_{1;t_2}^{\pm1},\cdots,\hat x_{k;t_2},\cdots,x_{n;t_2}^{\pm1}].$$

Let $(\mathbf{x}_{u},{\bf F}_{u})=\mu_k(\Sigma_{t_2})$ and $(\mathbf{x}_{v},{\bf F}_{v})=\mu_k(\Sigma_{t_1})$.
From the definition of the LP mutation, we know
\begin{eqnarray}
x_{i;u}=\begin{cases}x_{i;t_2}&\text{if }i\neq k\\ \hat{F}_{k;t_2}/x_{k;t_2}&\text{if } i=k\end{cases}\text{  and }x_{i;v}=\begin{cases}x_{i;t_1}&\text{if }i\neq k\\ \hat{F}_{k;t_1}/x_{k;t_1}&\text{if } i=k\end{cases}.\nonumber
\end{eqnarray}
Thus $x_{k;u}=\hat{F}_{k;t_2}(x_{1;t_2},\cdots,\hat x_{k;t_2},\cdots,x_{n;t_2})/x_{k;t_2}=\frac{\hat{F}_{k;t_2}(x_{1;t_1},\cdots,\hat x_{k;t_1},\cdots,x_{n;t_1})}{L_1x_{k;t_1}^{\pm1}+L_0}$.
From the definition of the exchange Laurent polynomial, we know the above equation has the form of
\begin{eqnarray}
x_{k;u}=\frac{F_{k;t_2}(x_{1;t_1},\cdots,\hat x_{k;t_1},\cdots,x_{n;t_1})}{L_1x_{k;t_1}^{\pm1}+L_0}M,
\end{eqnarray}
where $M$ is a Laurent monomial in $R[x_{1;t_1}^{\pm1},\cdots,\hat x_{k;t_1},\cdots,x_{n;t_1}^{\pm1}]$.
The above equation is just  the expansion of $x_{k;u}$ with respect to $\mathbf{x}_{t_1}$.
By Laurent phenomenon, and the fact $x_{k;t_1}\notin F_{k;t_2}(x_{1;t_1},\cdots,\hat x_{k;t_1},\cdots,x_{n;t_1})$, we obtain that $L_0=0$ and $\frac{F_{k;t_2}(x_{1;t_1},\cdots,\hat x_{k;t_1},\cdots,x_{n;t_1})}{L_1}$ is a Laurent polynomial in $R[x_{1;t_1}^{\pm1},\cdots,\hat x_{k;t_1},\cdots,x_{n;t_1}^{\pm1}]$.
Thus we have that $x_{k;t_2}=L_1x_{k;t_1}^{\pm 1}$ and $x_{k;u}$ has the form of  $x_{k;u}=\tilde M x_{k;t_1}^{\mp 1}$, where $\tilde M$ is a Laurent polynomial in $R[x_{1;t_1}^{\pm1},\cdots,\hat x_{k;t_1},\cdots,x_{n;t_1}^{\pm1}]$.

We claim that $\frac{F_{k;t_2}(x_{1;t_1},\cdots,\hat x_{k;t_1},\cdots,x_{n;t_1})}{L_1}$ is actually a Laurent monomial, i.e., $\tilde M$ is a Laurent monomial in $R[x_{1;t_1}^{\pm1},\cdots,\hat x_{k;t_1},\cdots,x_{n;t_1}^{\pm1}]$.

{\bf Case (i):} $x_{k;t_2}=L_1x_{k;t_1}$. Then $x_{k;t_1}=L_1^{-1}x_{k;t_2}$, which is the expansion of $x_{k;t_1}$ with respect to $\mathbf{x}_{t_2}$. By Laurent phenomenon, we can get that $L_1$ is a Luarent monomial in $$R[x_{1;t_1}^{\pm1},\cdots,\hat x_{k;t_1},\cdots,x_{n;t_1}^{\pm1}].$$ Then by Lemma \ref{mainlem}, $L_1$ is a unit in $R$ and $[\Sigma_{t_1}]=[\Sigma_{t_2}]$.

{\bf Case (ii):} $x_{k;t_2}=L_1x_{k;t_1}^{-1}$, in this case, $x_{k;u}=\tilde Mx_{k;t_1}$. By the same argument in case (i), we can get that $\tilde M$ is a Luarent monomial in  $R[x_{1;t_1}^{\pm1},\cdots,\hat x_{k;t_1},\cdots,x_{n;t_1}^{\pm1}]$. Then by Lemma \ref{mainlem}, $\tilde M$ is a unit in $R$ and  $[\Sigma_{t_1}]=[(\mathbf{x}_{u},\mathbf{F}_{u})]=[\mu_k(\Sigma_{t_2})]=\mu_k([\Sigma_{t_2}])$.
\end{proof}

\begin{remark}
The same method also works for cluster algebras and one can get the similar result.
\end{remark}

\section{On upper and lower bound of LP algebras}
The following definitions are natural generalizations of the correspondent notions of cluster algebras in \cite{CA3}.

For $i\in [1,n]$, we define the adjacent cluster $\mathbf{x}_i$ by $\mathbf{x}_i=(\mathbf{x}-\{x_i\})\cup \{x_i'\}$ where the cluster variables $x_i$ and $x_i'$ are related by exchange Laurent polynomial $\hat{F}_i$. Let $R[\mathbf{x}^{\pm 1}]$ be the ring of Laurent polynomials in $x_1,\dots,x_n$ with coefficients in $R$.

\begin{definition}
The \textbf{upper bound} $\mathcal{U}(\Sigma)$ and \textbf{lower bound} $\mathcal{L}(\Sigma)$ associated with a LP seed $\Sigma=(\mathbf{x,F})$ is defined by
$$\mathcal{U}(\Sigma)=R[\mathbf{x}^{\pm 1}]\cap R[\mathbf{x}_1^{\pm 1}]\cap\cdots \cap R[\mathbf{x}_n^{\pm 1}],\ \mathcal{L}(\Sigma)=R[x_1,x_1',\dots,x_n,x_n']$$
\end{definition}

Thus, $\mathcal{L}(\Sigma)$ is the $R$-subalgebra of $\mathcal{F}$ generated by the union of $n+1$ clusters $\mathbf{x}^{\pm 1},\mathbf{x}_1^{\pm 1},\dots,\mathbf{x}_n^{\pm 1}$. Note that $\mathcal{L}(\Sigma)\subseteq\mathcal{A}(\Sigma)\subseteq\mathcal{U}(\Sigma)$.

\subsection{Upper bound as invariant under LP mutation}\quad

For any LP seed $\Sigma=(\mathbf{x,F})$, the following two lemmas hold parallel to the correspondent results in \cite{CA3}.
\begin{lemma}\label{4.1}
\begin{equation}\label{u1}
\mathcal{U}(\Sigma)=\bigcap\limits_{j=1}^{n}R[x_1^{\pm1},\dots,x_{j-1}^{\pm1},x_j,x_j',x_{j+1}^{\pm1},\dots,x_n^{\pm1}].
\end{equation}
\end{lemma}
\begin{proof}
It is sufficient to show that $$R[\mathbf{x}^{\pm1}]\cap R[\mathbf{x}_1^{\pm1}]=R[x_1,x_1',x_2^{\pm1},\dots,x_n^{\pm1}].$$ The inclusion $\supseteq$ is clear, we only need to prove the converse inclusion.

For any $ y\in R[\mathbf{x}^{\pm1}]\cap R[\mathbf{x}_1^{\pm1}]$, $y$ is of the form $y=\sum\limits_{m=-M}^{N}c_mx_1^m$, where $M,N\in \mathbb{Z}_{\geq0}$ and $c_m\in R[x_2^{\pm1},\dots,x_n^{\pm1}]$.
If $M\geq0$, it is easy to see that $$y\in R[x_1,x_2^{\pm1},\dots,x_n^{\pm1}]\subseteq R[x_1,x_1',x_2^{\pm1},\dots,x_n^{\pm1}].$$

If $MN\neq0$, from the definition of LP seeds, $x_1\notin \hat{F}_1$, then
$$y|_{x_1\leftarrow\frac{\hat{F}_1}{x_1'}}=\sum\nolimits_{m=-M}^{N}c_m(\frac{\hat{F}_1}{x_1'})^m
=\sum\nolimits_{m=1}^{M}c_{-m}\hat{F}_1^{-m}x_1'^{m}+\sum\nolimits_{m=0}^{N}c_{m}\hat{F}_1^{m}x_1'^{-m}.$$
Since $y\in R[\mathbf{x}_1^{\pm1}]$, $y$ can be written as $\sum\limits_{p=M'}^{N'}c_px_1^p$ where $c_p\in R[x_2^{\pm1},\dots,x_n^{\pm1}]$, then we have $c_{-m}\hat{F}_1^{-m}\in R[x_2^{\pm1},\dots,x_n^{\pm1}]$. Thus,
$$y=\sum\limits_{m=1}^{M}c_{-m}\hat{F}_1^{-m}x_1'^{m}+\sum\limits_{m=0}^{N}c_mx_1^m\in R[x_1,x_1',x_2^{\pm1},\dots,x_n^{\pm1}].$$

If $N=0$, by similar discussion, we have $y\in R[x_1',x_2^{\pm1},\dots,x_n^{\pm1}]\subseteq R[x_1,x_1',x_2^{\pm1},\dots,x_n^{\pm1}]$.
\end{proof}

\begin{corollary}\label{4.2}
For $j\in [1,n]$, $y\in R[x_1^{\pm1},\dots,x_{j-1}^{\pm1},x_j,x_j',x_{j+1}^{\pm1},\dots,x_n^{\pm1}]$ if and only if $y$ is of the form $y=\sum\limits_{m=-M}^{N}c_mx_j^m$ where $M,N\in \mathbb{Z}_{\geq0}$, $c_m\in R[x_1^{\pm1},\dots,\hat{x_j}^{\pm1},\dots,x_n^{\pm1}]$ and $c_{-m}$ is divisible by $\hat{F}_j^m$ in $R[x_1^{\pm1},\dots,\hat{x_j}^{\pm1},\dots,x_n^{\pm1}]$ for $m\in [1,M]$.
\end{corollary}

\begin{lemma}
Suppose that $\hat{F}_j=F_j$ for $j\in[1,2]$, then $R[x_1,x_2^{\pm1}]\cap R[x_1^{\pm1},x_2,x_2']=R[x_1,x_2,x_2'].$
\label{lem1}
\end{lemma}
\vspace{-0.7cm}
\begin{proof}
The inclusion $\supseteq$ is clear, we only need to prove the converse inclusion.
For $y \in R[x_1,x_2^{\pm1}]\cap R[x_1^{\pm1},x_2,x_2']$, $y$ is of the form $y=\sum\limits_{m\in \mathbb{Z}}x_1^m(c_m+c_m'(x_2)+c_m''(x_2'))$, where $c_m \in R$, $c_m'(x_2)$ and $c_m''(x_2')$ are polynomials over $R$ without constant terms.

Let $M$ be the smaller integer such that $c_M+c_M'(x_2)+c_M''(x_2')\neq0$. If $M\geq0$, then it is easy to see that $y\in R[x_1,x_2,x_2']$.

Otherwise, the Laurent expression of $y$ is $\sum\limits_{m\in \mathbb{Z}}x_1^m(c_m+c_m'(x_2)+c_m''(\frac{F_2}{x_2}))$ by the assumption. Let $r_2$ be the sum of monomials in $F_2$ without $x_1$. Then there are nonzero terms with smallest power of $x_1$ in the Laurent expression of $y$, which are $x_1^M(c_M+c_M'(x_2)+c_M''(\frac{r_2}{x_2}))\neq0$, which contradicts the condition that $y\in R[x_1,x_2^{\pm1}]$.
\end{proof}

\begin{lemma}
Suppose that $\hat{F}_j=F_j$ for $j\in[1,2]$, then $$R[x_1,x_1',x_2^{\pm1}]=R[x_1,x_1',x_2,x_2']+R[x_1,x_2^{\pm1}].$$
\label{lem2}
\end{lemma}
\begin{proof}
The inclusion $\supseteq$ is clear, we only need to prove the converse inclusion. It is enough to show that $\forall M,N>0$, $x_1'^Nx_2^{-M}\in R[x_1,x_1',x_2,x_2']+R[x_1,x_2^{\pm1}]$.

By the assumption, we have $x_2x_2'=\hat{F}_2=F_2=g(x_1)+r_2$, where $g(x_1)=\sum\limits_{i=1}^{m}g_ix_1^i$, $g_i\in R$ and $r_2\neq 0\in R$ since $F_2$ is not divisible by $x_1$. If $g(x_1)=0$, then $x_2^{-1}=r_2^{-1}x_2'$, which implies that $x_1'^Nx_2^{-M}\in R[x_1,x_1',x_2,x_2']$.

Otherwise, let $p(x_1)=-\frac{g(x_1)}{r_2}\in R[x_1]$, then $x_2x_2'=g(x_1)+r_2$ can be written as
$$x_2^{-1}=r_2^{-1}x_2'+p(x_1)x_2^{-1}.$$
Repeatedly substituting $x_2^{-1}$ in the RHS of the above equation by $r_2^{-1}x_2'+p(x_1)x_2^{-1}$, we obtain $x_2^{-1}=P(x_1,x_2')+p^N(x_1) x_2^{-1}$, where $P(x_1,x_2')=r_2^{-1}x_2'\sum\limits_{i=0}^{N-1}p^i(x_1)\in R[x_1,x_2']$.

Then we have
\begin{align}\label{x1x2'}
x_1'^Nx_2^{-M}&=x_1'^NP^M(x_1,x_2')+x_1'^Np^{MN}(x_1) x_2^{-M}\nonumber\\
&+x_1'^N\sum\nolimits_{i=1}^{M-1}\binom{M}{i}(P(x_1,x_2'))^{M-i}(p(x_1)^Nx_2^{-1})^i,
\end{align}
where the first term of (\ref{x1x2'}) that is, $x_1'^NP^M(x_1,x_2')\in R[x_1,x_1',x_2']$.

For $p(x_1)=-\frac{1}{r_2}g(x_1)=-\frac{1}{r_2}\sum\limits_{i=1}^{m}g_ix_1^i$,  the smallest power of $x_1$ in $p^N(x_1)$ is $N$ and the greatest is $Nm$.
Thus we can rewrite $p^N(x_1)$ in the form $x_1^N(\sum\limits_{i=0}^{N(m-1)}p_ix_1^i)$ where $p_i\in R$, implying that for any integer $K>0$, we have $p^{NK}(x_1)\in x_1^NR[x_1]$. Since $x_1x_1'=\hat{F}_1=F_1\in R[x_2]$, we have $x_1'^Np^{NK}(x_1)\in R[x_1,x_2]$.

Then the middle term of (\ref{x1x2'}) is obvious in $R[x_1,x_2^{\pm1}]$, and the last term of (\ref{x1x2'}) is equal to $x_1'^N\sum\limits_{i=1}^{M-1}\binom{M}{i}(P(x_1,\frac{F_2}{x_2}))^{M-i}(p(x_1)^N x_2^{-1})^i\in R[x_1,x_2^{\pm1}]$.

Thus we finish the proof.
\end{proof}

\begin{proposition}
Suppose that $n\geq2$ and $\hat{F}_j=F_j$ for $j\in[1,n]$, then
\begin{equation}\label{u2}
\mathcal{U}(\Sigma)=\bigcap\limits_{j=2}^{n}R[x_1,x_1',x_2^{\pm1},\dots,x_{j-1}^{\pm1},x_j,x_j',x_{j+1}^{\pm1},\dots,x_n^{\pm1}].
\end{equation}
\label{U}
\end{proposition}

\begin{proof}
Comparing  (\ref{u1}) with (\ref{u2}), it is sufficient to show that $$R[x_1,x_1',x_2,x_2',x_3^{\pm1},\dots,x_n^{\pm1}]=R[x_1,x_1',x_2^{\pm1},\dots,x_n^{\pm1}]\cap R[x_1^{\pm1},x_2,x_2',x_3^{\pm1},\dots,x_n^{\pm1}].$$
Freeze the cluster variables $x_3,\dots,x_n$ and view $R[x_3^{\pm1},\dots,x_n^{\pm1}]$ as the new ground ring $R$, then the above equality reduces to
\begin{equation}\label{u2pf}
R[x_1,x_1',x_2,x_2']=R[x_1,x_1',x_2^{\pm1}]\cap R[x_1^{\pm1},x_2,x_2'].
\end{equation}

Suppose $F_1=f(x_2)+r_1$, $F_2=g(x_1)+r_2$, where $r_1\neq0,r_2\neq0\in R$ and $f(x_2),g(x_1)$ are polynomials over $R$ without constant terms. It is easy to see that Lemma \ref{lem1} and Lemma \ref{lem2} hold for four cases which are: (\emph{C1}) $x_2\notin F_1$ and $x_1\notin F_2$, that is, $f(x_2)=0$ and $g(x_1)=0$; (\emph{C2}) $x_2\in F_1$ and $x_1\in F_2$; (\emph{C3}) $x_2\notin F_1$ but $x_1\in F_2$; (\emph{C4}) $x_2\in F_1$ and $x_1\notin F_2$. Combining Lemma \ref{lem1} and Lemma \ref{lem2} with the fact that $R[x_1,x_1',x_2,x_2']\subseteq R[x_1^{\pm1},x_2,x_2']$, we obtain:
\begin{align}
  R[x_1,x_1',x_2^{\pm1}]\cap R[x_1^{\pm1},x_2,x_2']\nonumber
  &= (R[x_1,x_1',x_2,x_2']+R[x_1,x_2^{\pm1}])\cap R[x_1^{\pm1},x_2,x_2'] \nonumber \\
  &= R[x_1,x_1',x_2,x_2']+(R[x_1,x_2^{\pm1}]\cap R[x_1^{\pm1},x_2,x_2']) \nonumber \\
  &= R[x_1,x_1',x_2,x_2']\nonumber
\end{align}
Thus we have (\ref{u2pf}).
\end{proof}

\begin{lemma}
For a LP seed $(\mathbf{x,F})$, let $x_2'$ and $x_2''$ be the cluster variables exchanged with $x_2$ in the LP seeds $\mu_2(\mathbf{x,F})$ and $\mu_2\mu_1(\mathbf{x,F})$ respectively, then
\begin{equation}\label{u3}
R[x_1,x_1',x_2,x_2',x_3^{\pm1},\dots,x_n^{\pm1}]=R[x_1,x_1',x_2,x_2'',x_3^{\pm1},\dots,x_n^{\pm1}].
\end{equation}
\label{x2''}
\end{lemma}
\begin{proof}
We can freeze the cluster variables $x_3,\dots,x_n$ and view $R[x_3^{\pm1},\dots,x_n^{\pm1}]$ as the new ground ring $R$. Then we will prove the following equality can be reduced from  (\ref{u3}):
$$R[x_1,x_1',x_2,x_2']=R[x_1,x_1',x_2,x_2''].$$

We first show that $x_2''\in R[x_1,x_1',x_2,x_2']$.

In (\emph{C1}), it is easy to see that $x_2''=rx_2'$ for certain $r\in R$, which implies $x_2''\in R[x_1,x_1',x_2,x_2']$.

In (\emph{C2}), let $(\mathbf{x',F'})=\mu_1(\mathbf{x,F})$, then $x_2''$ is obtained by $x_2x_2''=\hat{F}_2'$.
Recall that $x_1x_1'=\hat{F}_1=F_1=f(x_2)+r_1$ and $x_2x_2'=\hat{F}_2=F_2=g(x_1)+r_2$, where $g(x_1)=\sum\limits_{i=1}^{m}g_ix_1^i$, $g_i\in R$ and $r_2\neq 0\in R$.
Because $F_2$ depends on $x_1$, from the definition of LP mutations, we have:
\begin{enumerate}
  \item $G_2=F_2|_{x_1\leftarrow N_2}=g(\frac{r_1}{x_1'})+r_2$.
  \item $H_2=G_2/c$, where $c$ is the product of all common factors of $g_ir_1^i$ for $i\in [1,m]$ and $r_2$.
  \item $F_2'=MH_2=x_1'^mH_2=h(x_1')+r_3$, \\
  where $r_3=\frac{g_mr_1^m}{c}$, $h(x_1')=\sum\limits_{i=1}^{m}h_ix_1'^i$, $h_j=
\begin{cases}
g_{m-i}r_1^{m-i}/c& j\in[1,m-1],\\
r_2/c& j=m
\end{cases}$.
\end{enumerate}

By Proposition \ref{2.7ofLPA}, there exist $x_2$ in $F_1'=F_1$ in $(\mathbf{x',F'})$, so that there is no $x_1'$ in $F_2'/\hat{F}_2'$, thus we have $\hat{F}_2'=F_2'$. It follows that
\begin{align}
  x_2x_2''&=\frac{1}{c}r_2x_1'^m+\sum\nolimits_{j=1}^{m-1}h_jx_1'^j+r_3 \nonumber\\
  &=\frac{1}{c}(x_2x_2'-g(x_1))x_1'^m+\sum\nolimits_{j=1}^{m-1}h_jx_1'^j+r_3 \nonumber\\
  &=x_2(\frac{1}{c}x_2'x_1'^m)-(\frac{1}{c}g(x_1)x_1'^m-(\sum\nolimits_{j=1}^{m-1}h_jx_1'^j)-r_3)\nonumber,
\end{align}
where $\frac{1}{c}g(x_1)x_1'^m=(\sum\limits_{i=1}^{m}\frac{g_i}{c}x_1^i)x_1'^m = \sum\limits_{i=1}^{m}\frac{g_i}{c}(x_1x_1')^ix_1'^{m-i} =\sum\limits_{i=1}^{m}\frac{g_i}{c}(f(x_2)+r_1)^ix_1'^{m-i}$.
Recall that $f(x_2)$ is a polynomial in $x_2$ without constant terms, then $\frac{g_m}{c}(f(x_2)+r_1)^m$ can be written as $x_2P_m+\frac{g_mr_1^m}{c}=x_2P_m+r_3$, where $P_m$ is a polynomial in $x_2$.

For $i\in[1,m-1]$, we have $\frac{g_i}{c}(f(x_2)+r_1)^ix_1'^{m-i}=x_2P_i+\frac{g_ir_1^i}{c}x_1'^{m-i}=x_2P_i+h_{m-i}x_1'^{m-i}$, where $P_i$ is a polynomial in $x_2$.

Then $\frac{1}{c}g(x_1)x_1'^m-(\sum\limits_{j=1}^{m-1}h_jx_1'^j)-r_3=x_2(\sum\limits_{i=1}^{m}P_i)$, which implies that $$x_2x_2''=x_2(\frac{1}{c}x_2'x_1'^m-\sum\limits_{i=1}^{m}P_i).$$ Thus $x_2''\in R[x_1,x_1',x_2,x_2']$.

For (\emph{C3}) and (\emph{C4}), it is enough to show for (\emph{C3}) by symmetry.
At this time, $F_2'$ is the same as that in (\emph{C2}).
Since $f(x_2)=0$, $(F_1'=F_1)|_{x_2\leftarrow F_2'/x_2''}=r_1$ is not divisible by $F_2'$, so that $\hat{F}_2'=F_2'$.
As a consequence, we have $x_2x_2''=x_2(\frac{1}{c}x_2'x_1'^m)$. Thus $x_2''\in R[x_1,x_1',x_2,x_2']$.

On the other hand, we can prove similarly that $x_2'\in R[x_1,x_1',x_2,x_2'']$. Then, (\ref{u3}) follows truely.
\end{proof}

\begin{theorem}\label{thma}
Assume that a LP seed $\Sigma=(\mathbf{x,F})$ satisfied $\hat{F}_j=F_j$ for $j\in[1,n]$ and $\Sigma'=(\mathbf{x',F'})$ is the LP seed obtained from the LP seed $\Sigma$  by mutation in direction $k$. Then the corresponding upper bounds coincide, that is, $\mathcal{U}(\Sigma)=\mathcal{U}(\Sigma')$.
\end{theorem}
\begin{proof}
Without loss of generality, we assume that $k=1$. Combining Proposition \ref{U} and Lemma \ref{x2''}, we finish the proof.
\end{proof}

\begin{proposition}\label{lem0}
If the exchange polynomials of a LP seed satisfy $\hat{F}_k=F_k$ for any $k\in [1,n]$, then $F_i\neq F_k$ for any $i\neq k$. Furthermore, any two of the exchange polynomials $\{F_k|k\in [1,n]\}$of a LP seed are coprime.
\end{proposition}
\begin{proof}
We will prove by contradiction. If $F_i=F_k$, then $$\hat{F}_i|_{x_k\leftarrow F_k/x_k'}=F_i|_{x_k\leftarrow F_k/x_k'}=F_k|_{x_k\leftarrow F_k/x_k'}=F_k$$ for $x_k\notin F_k$, implying that $F_k$ divides $\hat{F}_i|_{x_k\leftarrow F_k/x_k'}$, which contradicts the definition of exchange Laurent polynomials.

Besides, since the irreducibility of exchange polynomials for LP seeds, we conclude that the exchange polynomials of a LP seed are coprime under the condition that $\hat{F}_k=F_k$ for any $k\in [1,n]$.
\end{proof}

\begin{remark}
  When a cluster seed is a LP seed, the coprimeness of the cluster seed is equivalent to the condition that $\hat{F}_k=F_k$ for any $k\in [1,n]$.
  \label{rmk-cor}
\end{remark}

\begin{example}
Consider the LP seed $(\mathbf{x,F})=\{(a,b+c),(b,a+c),(c,a+(a+1)b)\}$ over $R=\mathbb{Z}$, which satisfies the condition that $\hat{F}_k=F_k$ for any $k\in \{a,b,c\}$.
Then the LP seed obtained by mutation at $b$ is $$(\mathbf{x',F'})=\{(a,1+d),(d,a+c),(c,a+d+1)\}$$ where $d=\frac{a+C}{b}$.
It is easy to see that $\hat{F_c'}=\frac{F_c'}{a}$, meaning that the condition that $\hat{F}_k=F_k$ for any $k$ for a LP seed may not hold under mutations.
\end{example}

\begin{definition}
Let $\Sigma=(\mathbf{x,F})$ be a LP seed, the \textbf{upper LP algebra} $\overline{\mathcal{A}}(\Sigma)$ defined by $\Sigma$  is the intersection of the subalgebras $\mathcal{U}(\Sigma')$ for all LP seeds $\Sigma'$ mutation-equivalent to $\Sigma$ .
\end{definition}

Theorem \ref{thma} has the following direct implication.
\begin{corollary}
Assume that all LP seeds mutation equivalent to a LP seed $\Sigma=(\mathbf{x,F})$ satisfy the condition that $\hat{F}_j=F_j$ for $j\in[1,n]$, then the upper bound $\mathcal{U}(\Sigma)$ is independent of the choice of LP seeds mutation-equivalent to $\Sigma$ , and so is equal to the upper LP algebra $\overline{\mathcal{A}}(\Sigma)$.
\end{corollary}

\subsection{On lower bound}
\subsubsection{A basis for lower bound}

\begin{definition}
Let $(\mathbf{x,F})$ be a LP seed. A \textbf{standard monomial} in $\{x_i,x_i'|i \in [1,n]\}$ is a monomial that contains no product of the form $x_ix_i'$.
\end{definition}

Let $\mathbf{x^{a}}=x_1^{a_1}\dots x_n^{a_n}$ be a Laurent monomial where $\mathbf{a}\in \mathbb{Z}^n$. For a Laurent polynomial in $x_1,\dots,x_n$, we order the each two terms $\mathbf{x^{a}}$ and $\mathbf{x^{a'}}$ lexicographically as follows:
\begin{equation}\label{lexicographically}
\mathbf{a}\prec\mathbf{a'}\ \text{if the first nonzero difference}\ a_j'-a_j\text{ is positive}.
\end{equation}
We set the term with the smallest lexicographical order as the first term in a Laurent polynomial.

\begin{theorem} \label{wuguan}
Assume that a LP seeds $\Sigma=(\mathbf{x,F})$ satisfies
\begin{enumerate}
  \item $\hat{F}_j=F_j$ for $j\in[1,n]$.
  \item in any $F_j$, the lexicographically first monomial is of the form $$\mathbf{x^{v_j}}=
  \begin{cases}
  x_{j+1}^{v_{j+1,j}}\cdots x_{n}^{v_{n,j}} & j\in[1,n-1] \\
  1 & j=n \end{cases}$$
   where $\mathbf{v_j}\in \mathbb{Z}_{\geq0}^{n-j}$ for $j\in[1,n-1]$.
\end{enumerate}
 Then the standard monomials in $x_1,x_1',\dots,x_n,x_n'$ form an $R$-basis for $\mathcal{L}(\Sigma)$.
\end{theorem}
\begin{proof}
The proof is using the same technique as in \cite{CA3}. We denote the standard monomials in $x_1,x_1',\dots,x_n,x_n'$ by $\mathbf{x^{(a)}}=x_1^{(a_1)}\cdots x_n^{(a_n)}$, where $\mathbf{a}=(a_1,\dots,a_n)\in \mathbb{Z}^n$ and
$$x_i^{(a_i)}=\begin{cases} x_i^{a_i}, & a_i\geq0 \\x_i'^{-a_i}, & a_i<0.\end{cases}$$
Note that $\mathbf{x^{(a)}}$ is a Laurent polynomial in $x_1,\dots,x_n$ and for any $i$, we have $x_i^{(-1)}=x_i'=x_i^{-1}\hat{F}_i=x_i^{-1}F_i$.
By the assumption for $F_i$, it follows that the lexicographically first monomial in $x_i^{(-1)}$ is $x_i^{-1}\mathbf{x^{v_i}}$, then the power of $x_i$ in $x_i^{(a_i)}$ is $a_i$ and there is no $x_1,\dots,x_{i-1}$ in $x_i^{(a_i)}$ .
Then the lexicographically first monomial in $\mathbf{x^{(a)}}$ is the product of $x_i^{a_i}(a_i>0)$ and $x_i^{a_i}(\mathbf{x^{v_i}})^{-a_i}(a_i<0)$.

We assume that $\mathbf{a}\prec\mathbf{a'}$ such that $a_i=a_i'$ for $i\in [1,k-1]$ and $a_k<a_k'$.
Let $P,\ M$ and $Q$ be the lexicographically first monomial of $\prod\limits_{j=1}^{k-1}x_j^{(a_j)}$, $x_k^{(a_k)}$ and $\prod\limits_{j=k+1}^{n}x_j^{(a_j)}$ respectively. Then the lexicographically first monomial of $\mathbf{x^{(a)}}$ is $PMQ$, similarly that of $\mathbf{x^{(a')}}$ is $P'M'Q'$.

Since $a_i=a_i'$ for $i\in [1,k-1]$, we have $P=P'=\prod\limits_{j=1}^{n}x_j^{p_j}$.

Since $a_k<a_k'$ and the power of $x_k$ in $x_k^{(a_k)}$ is $a_k$ and there is no $x_{i}\ (i\in [1,k-1])$ in $x_k^{(a_k)}$, we obtain $M=x_k^{a_k}\prod\limits_{j=k+1}^{n}x_j^{m_j}$ and $M'=x_k^{a_k'}\prod\limits_{j=k+1}^{n}x_j^{m_j'}$.

And $Q=(\prod\limits_{j=k+1,a_j>0}^{n}x_j^{a_j})(\prod\limits_{j=k+1,a_j<0}^{n}x_j^{a_j}(\mathbf{x^{v_j}})^{-a_j})=\prod\limits_{j=k+1}^{n}x_j^{q_j}$, similarly $Q'=\prod\limits_{j=k+1}^{n}x_j^{q_j'}$.

It follows that $$PMQ=(\prod\limits_{j=1}^{k-1}x_j^{p_j})(x_k^{p_k+a_k})(\prod\limits_{j=k+1}^{n}x_j^{p_j+m_j+q_j}),\ P'M'Q'=(\prod\limits_{j=1}^{k-1}x_j^{p_j})(x_k^{p_k+a_k'})(\prod\limits_{j=k+1}^{n}x_j^{p_j+m_j'+q_j'}).$$

Thus $PMQ\prec P'M'Q'$, implying that
\begin{equation}\label{lex1st}
\text{if }\mathbf{a}\prec\mathbf{a'}\text{, the lexicographically first monomial of }\mathbf{x^{(a)}}\prec \text{that of } \mathbf{x^{(a')}}.
\end{equation}
The linearly independence of standard monomials over $R$ follows at once from (\ref{lex1st}).
Since the product $x_ix_i'$ for any $i$ equals to $\hat{F_i}=F_i$, which is the linear combination of standard monomials in $x_1,x_1',\dots,x_n,x_n'$.
Thus they form a basis for $\mathcal{L}(\Sigma)$ .
\end{proof}


\subsubsection{Lower and upper bound}\quad

In the following statements, we always assume that $\Sigma=(\mathbf{x,F})$ is a LP seed of rank $n$ satisfying Condition \ref{condition}.

\begin{Notation}
We denote by $\bm{\varphi}:R[x_2,x_2',\dots,x_n,x_n']\rightarrow R[x_2^{\pm1},\dots,x_n^{\pm1}]$ the algebra homomorphism defined as the composition $\varphi_2\circ\varphi_1$, where
\begin{align}
  \varphi_1 & : R[x_2,x_2',\dots,x_n,x_n']\rightarrow R[x_1,x_2^{\pm1},\dots,x_n^{\pm1}]\text{ by }x_i\mapsto x_i\text{ and }x_i'\mapsto F_i/x_i.\nonumber\\
  \varphi_2 & : R[x_1,x_2^{\pm1},\dots,x_n^{\pm1}]\rightarrow R[x_2^{\pm1},\dots,x_n^{\pm1}]\text{ by }x_1\mapsto 0\text{ and }x_i^{\pm1}\mapsto x_i^{\pm1}.\nonumber
\end{align}
\end{Notation}

We denote by $R^{st}[x_2,x_2',\dots,x_n,x_n']$ (resp. $R^{st}[x_1,x_2,x_2',\dots,x_n,x_n']$) the $R$-linear span (resp. $R[x_1]$-linear span) of the standard monomials in $x_2,x_2',\dots,x_n,x_n'$.
\begin{lemma}
$R[x_2,x_2',\dots,x_n,x_n']=ker(\varphi)\oplus R^{st}[x_2,x_2',\dots,x_n,x_n']$.
\end{lemma}
\begin{proof}
For any $y\in R[x_2,x_2',\dots,x_n,x_n']$, replace $x_ix_i'\in y$ with $F_i$, then $y\in R^{st}[x_1,x_2,x_2',\dots,x_n,x_n']$. Thus we have $R[x_2,x_2',\dots,x_n,x_n']\subseteq R^{st}[x_1,x_2,x_2',\dots,x_n,x_n']$. It follows that
$$R[x_2,x_2',\dots,x_n,x_n']=ker(\varphi)+ R^{st}[x_2,x_2',\dots,x_n,x_n'].$$
Similarly using the tool of the proof of Theorem \ref{wuguan}, For $\mathbf{x^{(a)}}\in R^{st}[x_2,x_2',\dots,x_n,x_n']$,
the lexicographically first monomial of $\varphi(x_j^{(a_j)})$ is a Laurent monomial in $x_j,x_{j+1},\dots,x_n$ whose the power of $x_j$ is $a_j$, implying that if $\mathbf{a}\prec\mathbf{a'}$,
then the lexicographically first monomial of $\varphi(\mathbf{x^{(a)}})$ precedes the one of $\varphi(\mathbf{x^{(a')}})$.

Then the restriction of $\varphi$ to $R^{st}[x_2,x_2',\dots,x_n,x_n']$ is injective.
\end{proof}

\begin{Notation}
Given a Laurent polynomial $y\in R[x_1^{\pm1},\dots,x_n^{\pm1}]$, we denote by $LT(y)$ as the sum of all Laurent monomials with the smallest power of $x_1$ in the Laurent expansion of $y$ with nonzero coefficient.
\end{Notation}

The following results parallel to Lemma 6.4 and 6.5 in \cite{CA3} can be obtained similarly.
\begin{lemma}\label{LTy}
Suppose that $y=\sum_{m=a}^{b}c_mx_1^m$ where $c_m\in R^{st}[x_2,x_2',\dots,x_n,x_n']$ and $c_a\neq0$, then $LT(y)=\varphi(c_a)x_1^a$.
\end{lemma}
\begin{lemma}\label{6.5}
$R[x_1,x_2^{\pm1},\dots,x_n^{\pm1}]\cap R[x_1^{\pm1},x_2,x_2',\dots,x_n,x_n']=R[x_1,x_2,x_2',\dots,x_n,x_n']$.
\end{lemma}

\begin{lemma}\label{6.6}
Im$(\varphi)=R[x_2,x_2^{(-)},\dots,x_n,x_n^{(-)}]$, where $x_j^{(-)}=\begin{cases}x_j', & \mbox{if } x_1\notin F_j \\x_j^{-1}, & \mbox{otherwise}\end{cases}.$
\end{lemma}
\begin{proof}
By  Condition \ref{condition}, we have $\varphi(x_j')=\begin{cases}x_j', & \mbox{if } x_1\notin F_j \\x_j^{-1}M_j, & \mbox{otherwise}\end{cases}.$
Thus the inclusion $\subseteq$ is clear.

Let $J$ be the set of indexes $j\in [2,n]$ satisfying $x_1\notin F_j$. We set $W_j=x_j^{-1}M_j$.
To prove the converse inclusion, it is enough to show that $x_j^{-1}\in$ Im($\varphi$) for $j\in [2,n]-J$.

For $\mathbf{m}=(m_2,\dots,m_n), \mathbf{l}=(l_2,\dots,l_n)\in \mathbb{Z}^{n-1}$, let $\mathbf{x^{m}}$ be a Laurent monomial in $R[x_2^{\pm1},\dots,x_n^{\pm1}]$. Moreover, we set $\mathbf{W^{l}}=\prod\limits_{j=2}^{n}W_j^{l_j}$.
Then we have any $\mathbf{x^{m}}$ can be written as $\mathbf{W^{l}}$ satisfying $$m_j=-l_j+\sum\limits_{2\leq i<j}v_{ji}l_i.$$

Define the multiplicative monoid $\mathcal{W}=\{\mathbf{x^{m}}=\mathbf{W^{l}} | l_i\geq0\ for\ i\in [2,n]\ and\ m_j\geq0\ for\ j\in J\}$.
Then $\mathbf{W^{l}}\in \mathcal{W}$ if and only if
\begin{equation}\label{W-equ}
(a)~l_k\geq0~\text{for}~k\in [2,n],~ \;\;(b)~ \sum\limits_{2\leq i<j}v_{ji}l_i\geq l_j\text{ for }j\in J.
\end{equation}

By the equivalence condition (\ref{W-equ}) of $\mathbf{W^{l}}\in \mathcal{W}$, we obtain $x_j^{-1}\in \mathcal{W}$ for $j\in [2,n]-J$, implying that it suffices to show that $\mathcal{W}\subseteq$ Im$(\varphi)$.

For $W=\mathbf{W^{l}}\in \mathcal{W}$, we prove that $W\in$ Im$(\varphi)$ by induction on the degree of $W$.
When $deg(W)=0$, we have $W=1\in R \subseteq$ Im$(\varphi)$.
Assume that $deg(W)>0$ and for any $W'\in \mathcal{W}$ such that $deg(W')<deg(W)$, then $ W'\in$ Im$(\varphi)$.

Let $j=max\{i|l_i>0\ in\ W\}$, then we have $W/W_j\in \mathcal{W}$ by the equivalence condition (\ref{W-equ}) of $\mathbf{W^{l}}\in \mathcal{W}$. As a consequence, $W/W_j\in$ Im$(\varphi)$ under the induction assumption.
If $j\in [2,n]-J$, then $W_j\in$ Im$(\varphi)$ so that $W=(W/W_j)W_j\in$ Im$(\varphi)$.

Otherwise, since $l_j>0$, there exist $i\in[2,j-1]$ such that $v_{ji}l_i>0$, where $v_{ji}\neq0$ implies that $x_j\in M_i$. Fix such an index $i$.
By (iv) of Condition \ref{condition}, $F_j=f_j(x_i)+M_j$ and $f_j(x_i)=\sum\limits_{t=1}^{s_j}r_tc_t$, where $s_j$ is the number of terms of $f_j(x_i)$, $r_t\in R$ and $c_t=\prod\limits_{p\in [2,n]-{j}}x_p^{\gamma_{pt}}$ satisfying $\gamma_{pt}\in \mathbb{Z}_{\geq0}$ and $\gamma_{it}\neq0$.

From the definition of LP mutations and Condition \ref{condition}(i)(ii), we have $x_j'=x_j^{-1}f_j(x_i)+W_j$.
By multiplying both sides of that equation by $W/W_j$, we have $$(W/W_j)x_j'=x_j^{-1}\sum\limits_{t=1}^{s_j}r_tc_t(W/W_j)+W.$$
Since $(W/W_j)x_j'\in$ Im$(\varphi)$, we only need to show that for $t\in [1,s_j]$, $x_j^{-1}c_t(W/W_j)\in$ Im$(\varphi)$.

Define $W'=W_{i}^{l_i'}\cdots W_{j}^{l_j'}$, where $l_i'=1$ and $l_p'=min\{l_p,\sum\limits_{i\leq q<p}v_{pq}l_q'\}$ for $p\in [i+1,j]$.
Because $W/W'=W_{2}^{l_2}\cdots W_{i-1}^{l_{i-1}}W_{i}^{l_i-1}W_{i+1}^{l_{i+1}-l_{i+1}'}\cdots W_{j}^{l_j-l_j'}$,
the equivalence condition (\ref{W-equ}) of $W/W'\in \mathcal{W}$ can be written as

(a) for $k\in[i,j]$, $l_k-l_k'\geq0$;

(b) for $k\in J$, $l_k-l_k'\leq\sum\limits_{2\leq h<k}v_{kh}(l_h-l_h')$ $\Leftrightarrow$ $-l_k'+\sum\limits_{2\leq h<k}v_{kh}l_h'\leq -l_k+\sum\limits_{2\leq h<k}v_{kh}l_h$.

The inequalities of (a) are immediate from the definition of $W'$ and the choice of $i$.
And for inequality (b), we discuss in several cases:
\begin{enumerate}
  \item if $k\in [2,i-1]$, (b) is equivalent to $0\leq -l_k+\sum\limits_{2\leq h<k}v_{kh}l_h$.
  \item if $k=i$, we have $\sum\limits_{2\leq h<i}v_{kh}l_h'=0$, (b) is equivalent to $-1\leq -l_k+\sum\limits_{2\leq h<k}v_{kh}l_h$.
  \item if $k\in [i+1,n]$, when $l_k'=l_k$, (b) is equivalent to $\sum\limits_{2\leq h<k}v_{kh}l_h'\leq\sum\limits_{2\leq h<k}v_{kh}l_h$, when $l_k'\leq l_k$, $l_k'=\sum\limits_{2\leq h<k}v_{kh}l_h'$, then LHS of (b) is zero.
\end{enumerate}
Since $W\in \mathcal{W}$ and inequalities of (a) hold, we have inequality (b) holds for $W/W'$.
Thus $W/W'$ belongs to $\mathcal{W}$ with $deg(W/W')<deg(W)$, so that $W/W'\in$ Im$(\varphi)$.

Then we have
\begin{align}
  x_j^{-1}c_t(W/W_j) &= W\cdot \prod\nolimits_{p\in [2,n]-{j}}x_p^{\gamma_{pt}}/\mathbf{x^{v_j}}\nonumber\\
  &= (W/W')\cdot (W'\cdot (x_2^{\gamma_{2t}}\cdots x_{j-1}^{\gamma_{j-1,t}})\cdot(x_{j+1}^{\gamma_{j+1,t}-v_{j+1,j}}\cdots x_{n}^{\gamma_{nt}-v_{nj}})) \nonumber\\
  &=(W/W')\cdot P\nonumber
\end{align}
The claim $x_j^{-1}c_t(W/W_j)\in$ Im$(\varphi)$ is a consequence of the statement that $P\in R[x_2,\cdots,x_n]$.
Indeed, $R[x_2,\cdots,x_n]\subseteq$ Im$(\varphi)$.

The only variable with negative power (namely, -1) in $W'$ is $x_i$, since
\begin{align}
  W' &= W_{i}^{1}W_{i+1}^{l_{i+1}'}\cdots W_{j}^{l_j'} \nonumber\\
  &= x_i^{-1}\cdot (x_{i+1}^{v_{i+1,i}-l_{i+1}'}\cdot x_{i+2}^{(\sum\limits_{i\leq h<i+2}v_{i+2,h}l_h')-l_{i+2}'}\cdots x_{j}^{(\sum\limits_{i\leq h<j} v_{j,h}l_h')-l_{j}'})\cdot(x_{j+1}^{\sum\limits_{i\leq h\leq j}v_{j+1,h}l_h'}\cdots x_{n}^{\sum\limits_{i\leq h\leq j}v_{nh}l_h'})\nonumber\\
  &= x_i^{-1}\cdot Q \cdot (x_{j+1}^{\delta_{j+1,t}}\cdots x_{n}^{\delta_{nt}})\nonumber
\end{align}
where $\delta_{pt}=\sum\limits_{i\leq h\leq j}v_{ph}l_h'$ for $p\in [j+1,n]$.
Then we have
$$P=Q\cdot (\prod_{q\in[2,i-1]\cup[i+1,j-1]}x_q^{\gamma_{qt}})\cdot x_i^{\gamma_{it}-1}\cdot (\prod_{p\in [j+1,n]}x_p^{\delta_{pt}+\gamma_{pt}-v_{pj}}).$$

For $i$ is the fixed index such that $\gamma_{it}\in \mathbb{Z}_{>0}$, $\gamma_{it}-1>0$, then the power of $x_i$ is nonnegative.

For $p\in[j+1,n]$, we have $$\delta_{pt}+\gamma_{pt}-v_{pj}=v_{pi}l_i'+\cdots+v_{pj}l_j'+\gamma_{pt}-v_{pj}\geq v_{pj}(l_j'-1).$$
From the definition of $W'$, we obtain $l_j'=min\{l_j,\sum\limits_{i\leq q<j}v_{jq}l_q'\}$, and it is easy to see that  $l_j\geq1$ and $\sum\limits_{i\leq q<j}v_{jq}l_q'=v_{ji}+\sum\limits_{i<q<j}v_{jq}l_q'\geq v_{ji}\geq 1$ by the choice of $i$ and $j$.
Then $l_j'\geq1$. Thus the power of $x_p$ is nonnegative. Hence the power of any cluster variable is nonnegative. It follows that $P\in R[x_2,\cdots,x_n]$.
\end{proof}

By the same technique as in \cite{CA3}, we give the following theorem.
\begin{theorem}\label{Thmc}
If a LP seed $\Sigma=(\mathbf{x,F})$ satisfying Condition \ref{condition}, $\mathcal{L}(\Sigma)=\mathcal{U}(\Sigma)$.
\end{theorem}
\begin{proof}
We apply the induction on $n$, that is, the rank of the LP seed.
When $n=1$, by Lemma \ref{4.1}, we have $\mathcal{L}(\Sigma)=R[x_1,x_1']=\mathcal{U}(\Sigma)$.
Assume that $n\geq2$ and the statement holds for all algebras of rank 2 to $n-1$.
Then we consider about rank $n$.

By Lemma \ref{4.1}, we have $$\mathcal{U}(\Sigma)=\bigcap\limits_{j=2}^{n}
R[x_1^{\pm1},\dots,x_{j-1}^{\pm1},x_j,x_j',x_{j+1}^{\pm1},\dots,x_n^{\pm1}]\bigcap R[x_1,x_1',x_{2}^{\pm1},\dots,x_n^{\pm1}].$$
For the seed $\Sigma'$ obtained from $\Sigma$  by freezing at $x_1$, by the induction assumption, we have $\mathcal{L}(\Sigma')=\mathcal{U}(\Sigma')$, that is, $\bigcap\limits_{j=2}^{n}R[x_1^{\pm1},\dots,x_{j-1}^{\pm1},x_j,x_j',x_{j+1}^{\pm1},\dots,x_n^{\pm1}]=R[x_1^{\pm1},x_2,x_2',\dots,x_n,x_n']$.
Then it is enough to show that
\begin{equation}\label{thmc}
R[x_1,x_1',x_{2}^{\pm1},\dots,x_n^{\pm1}]\cap R[x_1^{\pm1},x_2,x_2',\dots,x_n,x_n']=R[x_1,x_1',\dots,x_n,x_n'].
\end{equation}
The inclusion $\supseteq$ is clear, we only need to prove the converse inclusion.

For $\forall y\in$ LHS of (\ref{thmc}), let $a$ be the smallest power of $x_1$ in $y|_{x_ix_i'\leftarrow F_i}$.
Then $y$ can be written as $\sum\limits_{m=a}^{b}c_mx_1^m$ where $c_m\in R^{st}[x_2,x_2',\dots,x_n,x_n']$.
By Lemma \ref{LTy}, we have $$LT(y)=\varphi(c_a)x_1^a\in R[x_1^{\pm1},x_{2}^{\pm1},\dots,x_n^{\pm1}].$$

If $a\geq0$, by Lemma \ref{6.5}, we have $y\in R[x_1,\dots,x_n,x_n'] \subseteq$ the RHS of (\ref{thmc}).

Otherwise, we apply the induction on $|a|$.
Since $y\in R[x_1,x_1',x_{2}^{\pm1},\dots,x_n^{\pm1}]$, by Lemma \ref{4.2} we have $\varphi(c_a)$ is divisible by $F_1^{|a|}$, that is $\varphi(c_a)=F_1^{|a|}z_a$ for certain $z_a\in R[x_2^{\pm1},\dots,x_n^{\pm1}]$.

When $J=\varnothing$, we have Im$(\varphi)=R[x_2^{\pm1},\dots,x_n^{\pm1}]$ according to Lemma \ref{6.6}. Then $z_a\in$ Im$(\varphi)$.

When $J\neq\varnothing$, we consider the LP seed $\Sigma^{\ast}$ obtained from $\Sigma$  by freezing at $\{x_j|j\in[2,n]-J\}$ and removing $x_1$.
In view of Lemma \ref{6.6}, we have $\mathcal{L}(\Sigma^{\ast})=$Im$(\varphi)$.
Besides, by the induction assumption, we have $\mathcal{L}(\Sigma^{\ast})=\mathcal{U}(\Sigma^{\ast})$.
Using Lemma \ref{4.1}, we obtain Im$(\varphi)=\bigcap\limits_{j\in J}R[x_2^{\pm1},\dots,x_j,x_j',\dots,x_n^{\pm1}]$.

For certain $j\in J$, $z_a$ can be written as $\sum\limits_{s\in \mathbb{Z}}c_sx_j^s$, where $c_s\in R[x_2^{\pm1},\dots,\hat{x_j},\dots,x_n^{\pm1}]$.
Since $F_1^{|a|}z_a=\sum\limits_{s\in \mathbb{Z}}(c_sF_1^{|a|})x_j^s\in$ Im$(\varphi)\subseteq R[x_2^{\pm1},\dots,x_j,x_j',\dots,x_n^{\pm1}]$, by Corollary \ref{4.2}  $c_sF_1^{|a|}$ is divisible by $F_j^{|s|}$ for $s<0$.
By Proposition \ref{lem0}, $F_1$ and $F_j$ are coprime, implying that $c_s$ is divisible by $F_j^{|s|}$.
Using Corollary \ref{4.2} again, we have $z_a\in R[x_2^{\pm1},\dots,x_j,x_j',\dots,x_n^{\pm1}]$.

By the arbitrariness of $j\in J$, we obtain $z_a\in$ Im$(\varphi)$.

Then there exist $c_a'\in R[x_2^{\pm1},\dots,x_n^{\pm1}]$ such that $z_a=\varphi(c_a')$.
It implies that $$LT(y)=\varphi(c_a)x_1^a=F_1^{|a|}z_ax_1^a=\varphi(c_a')F_1^{|a|}x_1^a=\varphi(c_a')x_1'^{|a|}.$$
Then we have $y=\sum\limits_{m=a}^{b}c_mx_1^m=\sum\limits_{m=a}^{-1}c_m'x_1'^{|m|}+\sum\limits_{m=0}^{b}c_mx_1^m\in R[x_1,x_1',\dots,x_n,x_n']$.
\end{proof}

\begin{corollary}
If a LP seed $\Sigma=(\mathbf{x,F})$ satisfied Condition \ref{condition}, then the standard monomials in $x_1,x_1',\dots,x_n,x_n'$ form an $R$-basis of the LP algebra $\mathcal{A}(\Sigma)$.
\end{corollary}
\begin{proof}
It is immediately from Theorem \ref{wuguan} and Theorem \ref{Thmc}.
\end{proof}

\begin{example}
Consider the LP seed $(\mathbf{x,F})=\{(a,bcd+1),(b,a+cd),(c,bd+1),(d,1+abc)\}$, it is easy to see that Condition \ref{condition} $(i)$ $(ii)$ $(iii)$ hold.
Since $M_c=1$ and $b|(F_c-M_c)$, $(iv)$ of Condition \ref{condition} holds.
Besides, $\varphi:\ b'\mapsto \frac{cd}{b},\ c'\mapsto \frac{bd+1}{c}=c',\ d'\mapsto \frac{1}{d}$,
Then it is clear that $d^{-1}\in$ Im$(\varphi)$ and $b^{-1}=\varphi(b'c'd'-d)\in$ Im$(\varphi)$.
Thus by Theorem \ref{Thmc}, we have $\mathcal{L}(\Sigma)=\mathcal{U}(\Sigma)$.

Note that this LP seed is not a cluster seed or a generalized cluster seed for $c\in F_a$ since $a\notin F_c$.
\end{example}

The cluster seed is acyclic if and only if there exist a permutation $\sigma$ such that for $i>j$, $b_{\sigma(i),\sigma(j)}\geq0$. Renumbering if necessary the indexes of the initial acyclic cluster, we assume that  for $i>j$, $b_{ij}\geq0$. Then by the exchange polynomials for cluster algebras, we conclude that  the cluster seed is acyclic if and only if for any $j$, $F_j=\frac{y_j}{1\oplus y_j}\prod\limits_{i>j}x_i^{b_{ij}}+\frac{y_j}{1\oplus y_j}\prod\limits_{i<j}x_i^{-b_{ij}}$.
\begin{proposition}
  Condition \ref{condition} is equivalent to acyclicity and coprimeness of exchange polynomials for a cluster seed which is a also LP seed.
\end{proposition}
\begin{proof}
  When a cluster seed is a LP seed, recall that (i) in Condition \ref{condition} is equivalent to coprimeness of exchange polynomials by Remark \ref{rmk-cor}.

  When a cluster seed satisfies the conditions (i) and (ii), for any $j\in [2,n-1]$, since $F_j$ is a binomial, we have $F_j=\mathbf{x^{v_j}}+\mathbf{x^{b_j}}$.
  If $x_i\in \mathbf{x^{b_j}}$ for some $i>j$, then $x_j\in \mathbf{x^{v_i}}$ for $b_{ji}b_{ij}<0$, which contradicts to the condition (ii).
  For $j=n$, $F_n=1+\mathbf{x^{b_n}}$.
  From the definition of exchange polynomials for cluster algebras, we have $\mathbf{x^{b_n}}$ is of the form $x_1^{|b_{1n}|}\cdots x_{n-1}^{|b_{n-1,n}|}$.
  For $j=1$, since for any $j>1$, we have $b_{j1}>0$ by the above discussion, so that we obtain $F_1=1+x_2^{|b_{12}|}\cdots x_{n}^{|b_{1n}|}$.
  Then the cluster seed satisfied the conditions (i) and (ii) is acyclic.
  Besides, it is easy to see that when a cluster seed is acyclic, it satisfies the conditions (i) and (ii).Thus the conditions (i) and (ii) are equivalent to acyclicity of exchange polynomials.

  Under the coprimeness and acyclicity, the cluster seed in fact satisfies (iii) and (iv) in Condition \ref{condition}.
\end{proof}

{\bf Acknowledgements:}\; {\em This project is supported by the National Natural Science Foundation of China  (No.12071422 and No.12131015) and the Zhejiang Provincial Natural Science Foundation of China (No.LY19A010023).}

\end{document}